\documentclass[12pt,reqno]{amsart}

\usepackage{amssymb}
\usepackage[dvipsnames]{xcolor}
\usepackage{helvet}
\usepackage{bm}
\usepackage[urlcolor=ForestGreen,citecolor=Maroon,ocgcolorlinks,bookmarksopen=True]{hyperref}
\usepackage{tikz}
\usetikzlibrary{calc}
\usepackage[left=2.6cm,right=2.6cm,top=2.7cm,bottom=2.8cm]{geometry}
\usepackage{multicol}
\usepackage{mathtools}

\newcommand{\PM}{M}
\newcommand{\charpoly}[1]{\mathcal{X}_{#1}}
 
\newcommand{\trimat}[5]{
#2 & #1                   \\
#1 & #3 & #1              \\
   & #1 & #4 & #1         \\
   &    & #1 & #5 &\ddots \\
&&&\ddots&\ddots
}
\newcommand{\diamat}[4]{
#1                \\
   & #2 &         \\
   &    & #3 &    \\
   &    &    & #4 \\
&&&&\ddots
}

\DeclareRobustCommand{\lrev}{\texorpdfstring{\reflectbox{\textup{L}}}{L}}

\newcommand{\setst}{\colon}

\newcommand{\sinB}[2][]{\sin#1\left(#2\right)}
\newcommand{\cosB}[2][]{\cos#1\left(#2\right)}

\newcommand{\ceil}[1]{\left\lceil#1\right\rceil}

\newcommand{\kdelta}[2]{\delta_{#1,#2}}

\newtheorem{theo}{Theorem}[section]
\newtheorem{lemma}[theo]{Lemma}
\newtheorem{coro}[theo]{Corollary}
\newtheorem{prop}[theo]{Proposition}

\theoremstyle{definition}
\newtheorem{exam}[theo]{Example}
\newtheorem{rem}[theo]{Remark}

\newtheorem{defn}[theo]{Definition}

\newtheorem{maintheo}{Theorem}

\title{Snake graphs and their characteristic polynomials}
\author{J. P. Bradshaw, P. Lampe, D. Ziga}
\date{25 October 2019}


\address{School of Mathematics, Statistics and Actuarial Science (SMSAS),
Sibson Building,
Parkwood Road,
University of Kent,
Canterbury CT2 7FS}

\begin{document}

\begin{abstract}
    The aim of the article is to understand the combinatorics of snake graphs by means of linear algebra. In particular, we apply Kasteleyn's and Temperley--Fisher's ideas about spectral properties of weighted adjacency matrices of planar bipartite graphs to snake graphs. First we focus on snake graphs whose set of turning vertices is monochromatic. We provide recursive sequences to compute the characteristic polynomials; they are indexed by the upper or the lower boundary of the graph and are determined by a neighbour count. As an application, we compute the characteristic polynomials for \lrev-shaped snake graphs and staircases in terms of Fibonacci product polynomials. Next, we introduce a method to compute the characteristic polynomials as convergents of continued fractions. Finally, we show how to transform a snake graph with turning vertices of two colours into a graph with the same number of perfect matchings to which we can apply the results above.
\end{abstract}

\maketitle

\section{Overview}

Snake graphs are planar bipartite graphs. They are created by putting together square pieces such that each piece is either right or above its predecessor.

Snake graphs are considered in the theory of cluster algebras where they are used to write down a formula for cluster variables in surface cluster algebras, see Musiker--Schiffler~\cite{MS10}, Musiker--Schiffler--Williams~\cite{MSW11}, and Propp~\cite{P05}. The summands in the formula are parametrised by perfect matchings of the given snake graph. \c{C}anak\c{c}{\i}--Schiffler~\cite{CS13,CS15,CS17} have studied snake graphs profoundly and noticed that the enumeration of perfect matchings is related to convergents of continued fractions.

In physics, perfect matchings are known as dimer models and Kasteleyn's \cite{K61,K63} and Temperley--Fisher's \cite{TF61} work on dimer statistics has led to useful counting formulas for perfect matchings. Specifically, we can assign weights to the edges of a planar bipartite graph and then read off the number of perfect matchings from the determinant of the weighted adjacency matrix.

Authors often emphasise the importance of spectral properties in graph theory. In particular, the determinant of a weighted adjacency matrix of a snake graph, and hence the number of perfect matchings, can be expressed as a product of the eigenvalues.

We are interested in the characteristic polynomials of weighted adjacency matrices of snake graphs. A special case of a theorem about grid graphs that Temperley--Fisher \cite{TF61} and Kasteleyn~\cite{K61,K63} discovered in independent work allows us to determine the characteristic polynomials for horizontal snake graphs. Here, the horizontal snake graph $H_n$ is created by putting together $n$ square tiles together in horizontal direction. The polynomials can be expressed in terms of the Fibonacci product polynomials $(P_r)_{r\geq 1}$ and $(Q_r)_{r\geq 1}$. Specifically, the characteristic polynomials turn out to be even and if we substitute $x=t^2$, then we may write $\charpoly{H_n}(t)=P_{r}(x)^2$ if $n=2r-1$ is odd, and  $\charpoly{H_n}(t)=(x-1)Q_r(x)^2$ if $n=2r$ is even. The first aim of the article is to express the characteristic polynomials for other interesting classes of snake graphs via Fibonacci product polynomials.

\begin{maintheo}[Theorem~\ref{Theo:L}, Theorem~\ref{Theo:Staircase}]
\begin{itemize}
    \item[(a)]
The \lrev-shaped snake graphs $\lrev_{r,s}$, which are created by placing first $r$ tiles to the right of each other and then $s$ tiles on top of each other, have the characteristic polynomials
\begin{align*}
\charpoly{\lrev_{2m,2n}}&=\left[Q_{m+n-1}-Q_{m-1}Q_{n-1}\right]\left[(x-1)Q_{m+n}+P_mP_n\right],\\
\charpoly{\lrev_{2m+1,2n+1}}&=\left[(x-1)Q_{m+n}-P_{m}P_{n}\right]\left[Q_{m+n+1}+Q_{m}Q_{n}\right],\\
\charpoly{\lrev_{2m,2n+1}}&=\left[P_{m+n}-Q_{m-1}P_{n}\right]\left[P_{m+n+1}+P_{m}Q_{n}\right].
\end{align*}
\newcommand{\tmpvarB}{\left(\left(x-3\right)^2\right)}
\item[(b)] The staircases $S_{m,3}$, which change direction from horizontal to vertical, or vice versa, at every other square tile, have characteristic polynomials
\begin{align*}
      \charpoly{S_{2k,3}}&=(x-2)Q_k\tmpvarB\left[(x-2)Q_k\tmpvarB+xQ_{k-1}\tmpvarB\right],\\
    \charpoly{S_{2k+1,3}}&=\left[Q_{k+1}\tmpvarB+(x-1)Q_k\tmpvarB\right]^2.
\end{align*}
\end{itemize}
\end{maintheo}

Let us colour the vertices of a snake graph in the two sets of the bipartition black and white. We call a square tile where the graph changes direction from horizontal to vertical, or vice versa, a turning tile. Every turning tile has a vertex with exactly $2$ neighbours and a vertex with exactly $4$ neighbours (opposite to each other) which we call turns.

Every rational number $x$ admits an essentially unique continued fraction expansion. Namely, we write $x$ as the sum of its integer part $\lfloor x\rfloor$ and its fractional part $x-\lfloor x\rfloor$ and iterate the procedure with the reciprocal of the fractional part. \c{C}anak\c{c}{\i}--Schiffler~\cite{CS17b} indicate a connection between snake graphs and continued fractions. For instance, the numerator of the rational number whose continued fraction expansion is encoded by the sign sequence of the snake graph equals the number of perfect matchings of the graph.

The rational number $x$ admits a second continued expansion. Namely, we write $x$ as the difference between the ceiling $\lceil x\rceil$ and the number $\lceil x\rceil-x\in [0,1)$ with whose reciprocal we iterate the process. The significance of the second method in the context of triangulated surfaces was first noted by Morier-Genoud--Ovsienko~\cite{MO19}.

\begin{maintheo}[Theorem~\ref{Theo:Numerator}]
\label{Theo:OverviewB}
Let $G$ be a snake graph whose turns are all black. For any Kasteleyn weighting the characteristic polynomial of the weighted adjacency matrix $A$ can be written as $\charpoly{A}(t)=\charpoly{B_1}(t^2)\charpoly{B_2}(t^2)$
where $\charpoly{B_1}(x)=p(x)$ is the numerator of the convergent of the continued fraction
\begin{align*}
    [[c_1(x),\ldots,c_k(x)]]=c_1(x)-\cfrac{1}{c_2(x) -\cfrac{1}{c_3(x) -\cfrac{1}{\ddots - \cfrac{1}{c_k(x)}}}}=\frac{p(x)}{q(x)}.
\end{align*}
Here, $1,2,\ldots,k$ is an enumeration of the black vertices in the upper boundary of $G$ and for every $l$ we put $c_l(x)=x-e_l$ where $e_l$ is the number of (white) vertices adjacent to $l$. The polynomial $\charpoly{B_2}(x)$ is constructed similarly using the vertices in the lower boundary.
\end{maintheo}

The polynomial $\charpoly{B_1}(x)$ is the characteristic polynomial of a matrix $B_1$ that describes interactions between the black vertices of the upper boundary and the white vertices of $G$. Theorem~\ref{Theo:OverviewB} implies $\charpoly{A}(0)=\pm \operatorname{det}(B_1)\operatorname{det}(B_2)$. The absolute value of this expression is equal to the square of the number of perfect matchings of $G$ thanks to the Theorem of Temperley--Fisher~\cite{TF61} and Kasteleyn~\cite{K61,K63}. We prove that it is enough to consider one of the boundary components to count matchings.

\begin{maintheo}[Special case of Theorem~\ref{Theo:Matchings}] In the setup of Theorem~\ref{Theo:OverviewB} the characteristic polynomials $\charpoly{B_1}(x)$ and $\charpoly{B_2}(x)$ have the same constant term, except for a possibly different sign. In particular, $\lvert\operatorname{det}(B_1)\rvert=\lvert\operatorname{det}(B_2)\rvert$ is equal to the number of perfect matchings of $G$.
\end{maintheo}

Finally, we show how to transform a snake graph with turns of two colours into a graph with the same number of perfect matchings to which we can apply the results above. The transformation is of combinatorial nature and constructed by rotations of subgraphs around tiles.

\section{Snake graphs and their adjacency matrices}

\subsection{Snake graphs} In this paper, unless stated otherwise, graphs are assumed to be simple, that is they are finite, undirected, and there is at most one edge between any two vertices. Recall that a graph is \emph{planar} if it can be drawn in the plane in such a way that edges intersect only at endpoints. A planar graph may admit many different embeddings into the plane; a planar graph together with a fixed embedding is called a \emph{plane} graph.

\begin{defn}[Snake graphs]
\label{Def:SnakeGraph}
\emph A \emph{snake graph} is a plane graph consisting of a sequence of tiles such that each tile is placed to the right or on top of the previous tile, starting from an initial tile. Here each tile is a square with four vertices and four edges and two adjacent tiles share exactly one edge.
\end{defn}

The procedure described in Definition \ref{Def:SnakeGraph} creates a strip of tiles that resembles a snake.

\begin{exam}
\label{Ex:SnakeGraphs}
In this note we are particularly interested in the following classes of snake graphs. Let $m$ and $n$ be natural numbers.
\begin{itemize}
    \item[(i)] The \emph{horizontal snake graph} $H_n$ is created by putting together $n$ square tiles in a horizontal direction. The example $H_6$ is shown on the left of Figure~\ref{Figure:Snake}. The graph $H_n$ is also known as the \emph{ladder graph} in the literature. The \emph{vertical snake graph} $V_n$ is the rotation of $H_n$ by $\pi/2$.
    \item[(ii)] The \emph{\lrev-shaped snake graph} $\lrev_{m,n}$ is created by placing first $m$ tiles to the right of each other and then $n-1$ tiles on top of each other. The snake graph $\lrev_{7,5}$ is shown in the middle of Figure~\ref{Figure:Snake}.
    \item[(iii)] The \emph{staircase} $S_{m,n}$ is the snake graph that we obtain by gluing together horizontal and vertical snake graphs $H_n$ and $V_n$ alternately along their first or last tiles, respectively, using a total number of $m$ pieces. For example, the staircase $S_{4,3}$ arises by gluing together four pieces $H_3$, $V_3$, $H_3$, $V_3$; it is shown in the right of Figure~\ref{Figure:Snake}.
\end{itemize}
\end{exam}

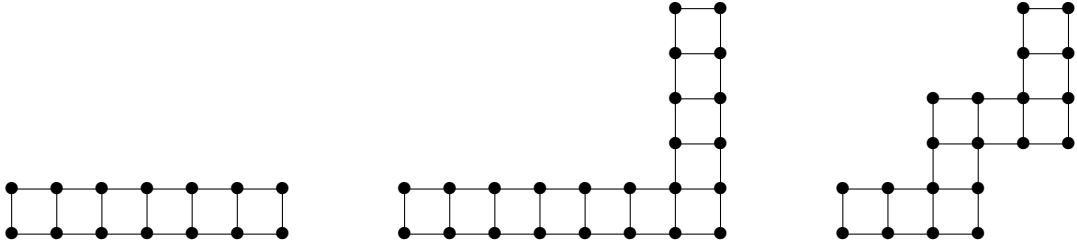
\begin{figure}[ht]
\begin{center}

%%% horizontal snake graph
\begin{tikzpicture}
\newcommand{\x}{0.6cm} %%% side length of a tile
\newcommand{\n}{7} %%% number of tiles +1

\foreach \k in {1,2,...,\n} {
\node at (\k*\x,0) {$\bullet$};
\node at (\k*\x,\x) {$\bullet$};
\path[draw] (\k*\x,0) to (\k*\x,\x);
}

\foreach \k in {2,3,...,\n} {
\path[draw] (\k*\x-\x,0) to (\k*\x,0);
\path[draw] (\k*\x-\x,\x) to (\k*\x,\x);
}

\end{tikzpicture}\hspace{1cm}
%%% L-shaped snake graph
\begin{tikzpicture}
\newcommand{\x}{0.6cm} %%% side length of a tile
\newcommand{\n}{8} %%% of number of horizontal tiles +1
\newcommand{\m}{5} %%% number of vertical tiles +1

\foreach \k in {1,2,...,\n} {
\node at (\k*\x,0) {$\bullet$};
\node at (\k*\x,\x) {$\bullet$};
\path[draw] (\k*\x,0) to (\k*\x,\x);
}

\foreach \k in {2,3,...,\n} {
\path[draw] (\k*\x-\x,0) to (\k*\x,0);
\path[draw] (\k*\x-\x,\x) to (\k*\x,\x);
}

\foreach \k in {2,3,...,\m} {
\node at (\n*\x,\k*\x) {$\bullet$};
\node at (\n*\x-\x,\k*\x) {$\bullet$};
\path[draw] (\n*\x,\k*\x) to (\n*\x-\x,\k*\x);
}

\path[draw] (\n*\x,\x) to (\n*\x,\m*\x);
\path[draw] (\n*\x-\x,\x) to (\n*\x-\x,\m*\x);

\end{tikzpicture}\hspace{1cm}
%%% zig zag
\begin{tikzpicture}
\newcommand{\x}{0.6cm} %%% side length of a tile

\node at (0,0) {$\bullet$};
\node at (0,\x) {$\bullet$};

\node at (\x,0) {$\bullet$};
\node at (\x,\x) {$\bullet$};

\node at (2*\x,0) {$\bullet$};
\node at (2*\x,\x) {$\bullet$};

\node at (3*\x,0) {$\bullet$};
\node at (3*\x,\x) {$\bullet$};

\node at (2*\x,2*\x) {$\bullet$};
\node at (3*\x,2*\x) {$\bullet$};

\node at (2*\x,3*\x) {$\bullet$};
\node at (3*\x,3*\x) {$\bullet$};

\node at (4*\x,2*\x) {$\bullet$};
\node at (4*\x,3*\x) {$\bullet$};

\node at (5*\x,2*\x) {$\bullet$};
\node at (5*\x,3*\x) {$\bullet$};

\node at (4*\x,4*\x) {$\bullet$};
\node at (4*\x,5*\x) {$\bullet$};

\node at (5*\x,4*\x) {$\bullet$};
\node at (5*\x,5*\x) {$\bullet$};

\path[draw] (0,0) to (3*\x,0);
\path[draw] (0,\x) to (3*\x,\x);

\path[draw] (2*\x,0) to (2*\x,3*\x);
\path[draw] (3*\x,0) to (3*\x,3*\x);

\path[draw] (2*\x,2*\x) to (5*\x,2*\x);
\path[draw] (2*\x,3*\x) to (5*\x,3*\x);

\path[draw] (4*\x,2*\x) to (4*\x,5*\x);
\path[draw] (5*\x,2*\x) to (5*\x,5*\x);

\path[draw] (0,0) to (0,\x);
\path[draw] (\x,0) to (\x,\x);

\path[draw] (4*\x,4*\x) to (5*\x,4*\x);
\path[draw] (4*\x,5*\x) to (5*\x,5*\x);

\end{tikzpicture}
\end{center}
\caption{A horizontal snake graph, an \lrev-shaped snake graph and a staircase}
\label{Figure:Snake}

\end{figure}

\subsection{Special vertices and edges in snake graphs}
\label{Subsec:SpecialVertices}

Let $G=(V,E)$ be a snake graph.

\begin{defn}[Turns]
\label{Def:Turns}
A vertex of $G$ is called a \emph{turn} if the following conditions hold.
\begin{itemize}
    \item[(i)] It is adjacent to exactly $2$ other vertices or it is adjacent to exactly $4$ other vertices.
    \item[(ii)] It is neither the lower right, the lower left nor the upper left vertex of the first tile; and it is neither the upper left, the upper right nor the lower right vertex of the last tile.
\end{itemize}
A turn is called a \emph{$2$-turn} if it is adjacent to exactly $2$ vertices and a \emph{$4$-turn} otherwise. A tile of $G$ having exactly two turns as vertices is called a \emph{turning tile}.
\end{defn}

The name \emph{turn} reflects the fact that the snake changes from horizontal direction to vertical direction, or vice versa, at tiles having a $2$-turn and a $4$-turn.

\begin{defn}[Internal and external edges]
\label{Def:Internal}
An edge of $G$ is called \emph{internal} if one of the following conditions holds:
\begin{itemize}
    \item[(i)] it is shared by two adjacent square tiles;
    \item[(ii)] it is the lower side of the first tile;
    \item[(iii)] it is the upper side of the last tile.
\end{itemize}
The edge is called \emph{external} otherwise.
\end{defn}

Hence, with two exceptions, the external edges are the edges that bound the infinite face of the graph. Note that every square tile is bounded by exactly $2$ internal and $2$ external edges. An example is shown on the left in Figure \ref{fig:SpecialEdges}.

We will refer to all edges that bound the infinite face (including two internal edges) as \emph{boundary edges}. 

Suppose that $G$ has at least two tiles. The \emph{start edge} of $G$ is the edge of the first square tile connecting the $2$ vertices that do not belong to any other tile. The \emph{end edge} is defined similarly for the last tile. The subgraph of $G$ formed by the boundary edges is a circular graph $C(G)$. We are interested in the following two full subgraphs of the infinite face.

\begin{defn}[Upper and lower boundary]
\label{Def:UpperLowerBoundary}
If we further remove the start and the end edge from the circular graph $C(G)$, then we obtain a disjoint union of two path graphs which we will call the \emph{upper} and \emph{lower boundary} of $G$. Here, the path graph containing the lower left vertex of the first tile is the upper boundary if the next tile is above the first tile, otherwise it is the lower boundary. Vertices in the upper or lower boundary are called \emph{upper} and \emph{lower} vertices, respectively.
\end{defn}

We can extend Definition~\ref{Def:UpperLowerBoundary} to the snake graph with $1$ tile by declaring the upper boundary to be the upper edge and the lower boundary to be the lower edge.

An example is shown on the right in Figure \ref{fig:SpecialEdges}.

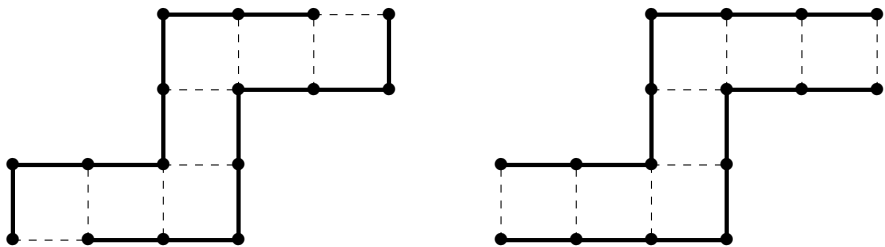
\begin{figure}

\begin{center}
\begin{tikzpicture}

\pgfdeclarelayer{background}
\pgfdeclarelayer{foreground}
\pgfsetlayers{background,main,foreground}

\definecolor{c1}{RGB}{150,150,150}

\newcommand{\x}{1cm} %%% side length of a tile
\newcommand{\R}{0.2pt}  %%% size of bullet
\newcommand{\s}{0.5cm}  %%% shift of labels
\newcommand{\sos}{0.22cm}  %%% shift of labels (signs)

\draw (0,0) node {$\bullet$};
\draw (0,\x) node {$\bullet$};

\draw (\x,0) node {$\bullet$};
\draw (\x,\x) node {$\bullet$};

\draw (2*\x,0) node {$\bullet$};
\draw (2*\x,\x) node {$\bullet$};

\draw (3*\x,0) node {$\bullet$};
\draw (3*\x,\x) node {$\bullet$};

\draw (2*\x,2*\x) node {$\bullet$};
\draw (3*\x,2*\x) node {$\bullet$};

\draw (2*\x,3*\x) node {$\bullet$};
\draw (3*\x,3*\x) node {$\bullet$};

\draw (4*\x,2*\x) node {$\bullet$};
\draw (4*\x,3*\x) node {$\bullet$};

\draw (5*\x,2*\x) node {$\bullet$};
\draw (5*\x,3*\x) node {$\bullet$};

\begin{pgfonlayer}{background}
\path[draw,dashed] (0,0) to (\x,0);
\path[draw,ultra thick] (\x,0) to (3*\x,0);
\path[draw,ultra thick] (0,\x) to (2*\x,\x);
\path[draw,dashed] (2*\x,\x) to (3*\x,\x);

\path[draw,dashed] (2*\x,\x) to (2*\x,0);
\path[draw,ultra thick] (2*\x,\x) to (2*\x,3*\x);
\path[draw,ultra thick] (3*\x,0) to (3*\x,2*\x);
\path[draw,dashed] (3*\x,2*\x) to (3*\x,3*\x);

\path[draw,ultra thick] (3*\x,2*\x) to (5*\x,2*\x);
\path[draw,dashed] (2*\x,2*\x) to (3*\x,2*\x);
\path[draw,ultra thick] (2*\x,3*\x) to (4*\x,3*\x);
\path[draw,dashed] (4*\x,3*\x) to (5*\x,3*\x);

\path[draw,dashed] (4*\x,2*\x) to (4*\x,3*\x);
\path[draw,ultra thick] (5*\x,2*\x) to (5*\x,3*\x);

\path[draw,ultra thick] (0,0) to (0,\x);
\path[draw,dashed] (\x,0) to (\x,\x);

\end{pgfonlayer}

\end{tikzpicture}\hspace{1cm}\begin{tikzpicture}

\pgfdeclarelayer{background}
\pgfdeclarelayer{foreground}
\pgfsetlayers{background,main,foreground}

\definecolor{c1}{RGB}{150,150,150}

\newcommand{\x}{1cm} %%% side length of a tile
\newcommand{\R}{0.2pt}  %%% size of bullet
\newcommand{\s}{0.5cm}  %%% shift of labels
\newcommand{\sos}{0.22cm}  %%% shift of labels (signs)

\draw (0,0) node {$\bullet$};
\draw (0,\x) node {$\bullet$};

\draw (\x,0) node {$\bullet$};
\draw (\x,\x) node {$\bullet$};

\draw (2*\x,0) node {$\bullet$};
\draw (2*\x,\x) node {$\bullet$};

\draw (3*\x,0) node {$\bullet$};
\draw (3*\x,\x) node {$\bullet$};

\draw (2*\x,2*\x) node {$\bullet$};
\draw (3*\x,2*\x) node {$\bullet$};

\draw (2*\x,3*\x) node {$\bullet$};
\draw (3*\x,3*\x) node {$\bullet$};

\draw (4*\x,2*\x) node {$\bullet$};
\draw (4*\x,3*\x) node {$\bullet$};

\draw (5*\x,2*\x) node {$\bullet$};
\draw (5*\x,3*\x) node {$\bullet$};

\begin{pgfonlayer}{background}
\path[draw,ultra thick] (0,0) to (3*\x,0);
\path[draw,ultra thick] (0,\x) to (2*\x,\x);
\path[draw,dashed] (2*\x,\x) to (3*\x,\x);

\path[draw,dashed] (2*\x,\x) to (2*\x,0);
\path[draw,ultra thick] (2*\x,\x) to (2*\x,3*\x);
\path[draw,ultra thick] (3*\x,0) to (3*\x,2*\x);
\path[draw,dashed] (3*\x,2*\x) to (3*\x,3*\x);

\path[draw,ultra thick] (3*\x,2*\x) to (5*\x,2*\x);
\path[draw,dashed] (2*\x,2*\x) to (3*\x,2*\x);
\path[draw,ultra thick] (2*\x,3*\x) to (5*\x,3*\x);

\path[draw,dashed] (4*\x,2*\x) to (4*\x,3*\x);
\path[draw,dashed] (5*\x,2*\x) to (5*\x,3*\x);

\path[draw,dashed] (0,0) to (0,\x);
\path[draw,dashed] (\x,0) to (\x,\x);

\end{pgfonlayer}

\end{tikzpicture}

\end{center}

\caption{The external edges (left) and the upper and lower boundary (right)}
\label{fig:SpecialEdges}
\end{figure}

\subsection{Perfect matchings and continued fractions}

Suppose $G=(V,E)$ is a simple graph with vertex set $V$ and edge set $E$.

\begin{defn}[Perfect matchings]
A \emph{perfect matching} of the graph $G$ is a subset $P\subseteq E$ of the set of edges such that every vertex of the graph is incident to exactly one edge in $P$. 
\end{defn}

Perfect matchings are also known as \emph{dimer models} in the literature. We denote the number of perfect matchings of $G$ by $\PM(G)$.

\begin{defn}[Fibonacci numbers]
The \emph{Fibonacci numbers} $(F_n)_{n\geq 0}$ are given by the sequence $F_0 = 0$, $F_1 = 1$ and $F_n = F_{n-1} + F_{n-2}$ for $n \geq 2$.
\end{defn}

The first Fibonacci numbers are shown in Figure \ref{fig:Fibonacci}.

\begin{exam}
\label{Ex:Fibonacci}
It is well-known and easy to show that the horizontal snake graph satisfies $\PM(H_n)=F_{n+2}$.
\end{exam}

\begin{figure}
    \centering
\begin{center}
\begin{tabular}{|c||c|c|c|c|c|c|c|c|c|c|c|}\hline
     $n$ & $0$ & $1$ & $2$ & $3$ & $4$ & $5$ & $6$ & $7$ & $8$ & $9$ & $10$\\\hline
     $F_n$ & $0$ & $1$ & $1$ & $2$ & $3$ & $5$ & $8$ & $13$ & $21$ & $34$ & $55$ \\\hline
\end{tabular}
\end{center}
    \caption{The Fibonacci numbers}
    \label{fig:Fibonacci}
\end{figure}

More generally, the number of perfect matchings of any snake graph can directly be given using continued fractions thanks to the work of \c{C}anak\c{c}{\i}--Schiffler~\cite{CS17b}. Given a snake graph, we place a sign on every internal edge and with the sign sequence we associate a sequence of natural numbers $(a_1,a_2,\ldots,a_n)$ by counting consecutive entries in the sequence
\begin{align*}
(\underbrace{+,\ldots,+}_{a_1}\,,\,\underbrace{-,\ldots,-}_{a_2}\,,\,\underbrace{+,\ldots,+}_{a_3}\,,\ldots).
\end{align*}
We refer the reader to Schiffler's overview article \cite{S19} for precise details of the construction of the sign sequence. Then we consider the continued fraction
\begin{align}
\label{Eqn:CFrac}
[a_1,a_2,\ldots,a_n ]=a_1+\cfrac{1}{a_2 +\cfrac{1}{a_3 +\cfrac{1}{\ddots + \cfrac{1}{a_n}}}}.
\end{align}
Let us write the continued fraction as the rational number
\begin{align}
\label{Eqn:Expansion}
[a_1,a_2,\ldots,a_n]=\frac{p}{q}
\end{align}
with coprime integers $p,q\geq 1$. Then the numerator $p$ gives us the number of perfect matchings of the snake graph. For example, the sign sequence of the horizontal snake graph $H_n$ is given by $a_1=a_{n-1}=2$ and $a_2=\ldots=a_{n-2}=1$.

\begin{rem}
\label{Rem:Unique}
The continued fraction expansion of the rational number $p/q$ in Equation (\ref{Eqn:Expansion}) is essentially unique. More precisely, given numbers $a_i \geq 1$ with $1\leq i\leq n$ and $b_i \geq 1$ with $1\leq i\leq m$ such that $[a_1,a_2,\ldots,a_n]= [b_1,b_2,\ldots,b_m]$. Then exactly one of the two cases holds.
\begin{itemize}
\item[(i)] We have $n=m$ and $a_i=b_i$ for all $1\leq i\leq n$.
\item[(ii)] We have $\lvert n-m\rvert =1$; without loss of generality let us assume $m=n+1$. Then the equality $a_i=b_i$ holds for all $1\leq i\leq n-1$ and the last entries satisfy $b_{n}=a_{n}-1$ and $b_{n+1}=1$.
\end{itemize}
It follows that the continued fraction expansion of a rational number as in Equation (\ref{Eqn:CFrac}) becomes unique if we impose the condition $n \equiv 0 \operatorname{mod} 2$. Notice that if the sequence $(a_i)$ is constructed from a snake graph $G$, then the sequence $(b_i)$ in case (ii) arises by using the sign of the right edge of the last tile instead of the sign of the upper edge.
\end{rem}

\newcommand\myeq{\stackrel{\mathclap{\normalfont\tiny \mbox{def}}}{=}}

The rational number $p/q$ admits another continued fraction expansion
\begin{align}
\label{Eqn:NegConFrac}
\frac{p}{q}=[[c_1,c_2,\ldots,c_k ]]&\myeq c_1-\cfrac{1}{c_2 -\cfrac{1}{c_3 -\cfrac{1}{\ddots - \cfrac{1}{c_k}}}},
\end{align}
which is unique if $c_i\geq 2$ for all $i\geq 1$. Note that
\begin{align*}
[[c_1,c_2,\ldots,c_k ]]=[c_1,-c_2,c_3,\ldots,(-1)^{k+1}c_k].
\end{align*}
Morier-Genoud and Ovsienko \cite{MO19}, based on an observation by Hirzebruch ~\cite[Equation (19)]{H73}, note that (assuming $n$ is even) both expansions are related by the formula
\begin{align}
\label{Eqn:TwoContinuedFractions}
(c_1,\ldots,c_k)=
\big(a_1+1,\underbrace{2,\ldots,2}_{a_2-1},\,
a_3+2,\underbrace{2,\ldots,2}_{a_4-1},a_5+2,\ldots,
a_{n-1}+2,\underbrace{2,\ldots,2}_{a_{n}-1}\big).
\end{align}

\subsection{Weighted adjacency matrices}

As before, let $G=(V,E)$ be a simple graph. We denote by $m=\lvert V\rvert$ the cardinality of the vertex set.

\begin{defn}[Weightings]
A map $w\colon E\to \{-1,1\}$, which assigns a weight $e\mapsto w(e)$ to every edge $e\in E$, is called a \emph{weighting} of the graph $G$.
\end{defn}

More generally, some authors consider weightings $e\colon E\to R$ with values in a general set $R$. In the context of Kasteleyn's and Temperley--Fisher's Theorem, which we will apply to count perfect matchings in bipartite graphs, weightings with real or complex numbers of absolute value $1$ are considered. In this article, we are mostly interested in real weightings, so we assume $R=\{-1,1\}$.

\begin{defn}[Weighted adjacency matrices] Fix a labelling $V=\{v_0,v_1,\ldots,v_{m-1}\}$ of the set of vertices.
\begin{itemize}
\item[(i)]
 The \emph{adjacency matrix} is the $m\times m$ matrix $Z=(z_{ij})_{0\leq i,j\leq m-1}\in \operatorname{Mat}_{m\times m}(\mathbb{Z})$ with entries
\begin{align*}
    z_{ij}=\begin{cases}
    1,&\textrm{if $i$ and $j$ are connected by an edge};\\
    0,&\textrm{if $i$ and $j$ are not connected by an edge}.
    \end{cases}
\end{align*}
\item[(ii)] Given a weighting $w\colon E\to \{\pm 1\}$, the \emph{weighted adjacency matrix} is the $m\times m$ matrix $A=(a_{ij})_{0\leq i,j\leq m-1}\in \operatorname{Mat}_{m\times m}(\mathbb{Z})$ with entries
\begin{align*}
    a_{ij}=\begin{cases}
    w(e),&\textrm{if $i$ and $j$ are connected by an edge $e$};\\
    0,&\textrm{if $i$ and $j$ are not connected by an edge}.
    \end{cases}
\end{align*}
\end{itemize}
\end{defn}

By construction, the adjacency matrix and the weighted adjacency matrix are symmetric matrices for every $G$ and every weighting $w$. In particular, $A$ will be diagonalisable and all eigenvalues will be real numbers. The eigenvalues of $A$ do not depend on the choice of the ordering of the vertices.

\subsection{The theorem of Temperley--Fisher and Kasteleyn}
\label{Subsec:Kasteleyn}

Recall the following definition.

\begin{defn}[Bipartite graphs]
A graph $G=(V,E)$ is called \emph{bipartite} if we can divide the vertices into two disjoint sets $V_1$ and $V_2$ such that every edge connects a vertex in $V_1$ to a vertex in $V_2$.
\end{defn}

We often call a vertex in $V_1$ \emph{black} and a vertex in $V_2$ \emph{white}. Note that snake graphs are bipartite, see Figure \ref{fig:weighting}.

\begin{defn}[Bipartite orders]
 Given a bipartite graph $G=(V_1\sqcup V_2,E)$, we abbreviate $m_i=\lvert V_i\rvert$ for $i\in \{1,2\}$. The pair $(m_1,m_2)$ is called the \emph{order} of the bipartite graph $G$.
\end{defn}

Let us fix a plane bipartite graph $G=(V_1\sqcup V_2,E)$ together with a weighting $w\colon E\to \{\pm 1\}$ for the rest of the section. Recall that the embedding of $G$ divides the plane into regions which are called \emph{faces}. There is one unbounded face which is called the \emph{infinite face}.

\begin{defn}[Kasteleyn weightings]
\label{Def:Weighting}
 The weighting $w$ is called a \emph{Kasteleyn weighting} if each face (including the infinite face) having the shape of a $k$-gon has$\ldots$
 \begin{itemize}
 \item[(i)] an odd number of negative signs if $k \equiv 0 \operatorname{mod} 4$;
 \item[(ii)] an even number of negative signs if $k \equiv 2
 \operatorname{mod} 4$.
 \end{itemize}
\end{defn}

It can be shown that every plane bipartite graph admits a Kasteleyn weighting. To construct weightings for snake graphs we distinguish external and internal edges, see Definition~\ref{Def:Internal}.

\begin{defn}[Real weightings for snake graphs]
\label{Def:ConstructionRealWeighting}
Given a snake graph $G=(V,E)$, we set $w(e)=1$ for every edge $e$ in the upper boundary and every edge in the lower boundary and we alternate the signs of the remaining edges, starting with $+1$ on the start edge.
\end{defn}

An example is shown in Figure \ref{fig:weighting}.

\begin{prop}
The weighting $w$ from Definition \ref{Def:ConstructionRealWeighting} is a Kasteleyn weighting.
\end{prop}

\begin{proof}
By construction, every internal face is bounded by $4$ edges and exactly one of the edges weighs $-1$. The infinite face has $2n+2$ edges, if $n$ denotes the number of tiles of the graph. All of its edges weigh $1$, except for the end edge of the last tile when $n$ is odd.
\end{proof}

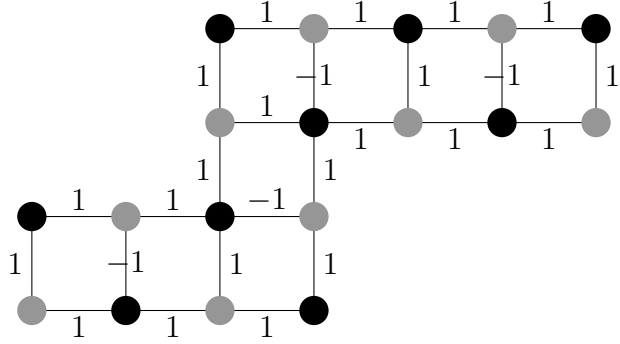
\begin{figure}

\begin{center}
\begin{tikzpicture}

\pgfdeclarelayer{background}
\pgfdeclarelayer{foreground}
\pgfsetlayers{background,main,foreground}

\definecolor{c1}{RGB}{150,150,150}

\newcommand{\x}{1.25cm} %%% side length of a tile
\newcommand{\R}{0.5pt}  %%% size of bullet
\newcommand{\s}{0.5cm}  %%% shift of labels
\newcommand{\sos}{0.22cm}  %%% shift of labels (signs)

\draw (0,0) node[fill,circle,minimum size=\R,color=c1]{};
\node at (0.5*\x,0*\x-\sos) {$1$};
\node at (0.5*\x,1*\x+\sos) {$1$};
\node at (0-\sos,0.5*\x) {$1$};
\draw (0,\x) node[fill,circle,minimum size=\R]{};

\draw (\x,0) node[fill,circle,minimum size=\R]{};
\node at (1.5*\x,0*\x-\sos) {$1$};
\node at (1.5*\x,1*\x+\sos) {$1$};
\draw (\x,\x) node[fill,circle,minimum size=\R,color=c1]{};
\node at (\x,0.5*\x) {$-1$};

\draw (2*\x,0) node[fill,circle,minimum size=\R,color=c1]{};
\node at (2*\x+\sos,0.5*\x) {$1$};
\draw (2*\x,\x) node[fill,circle,minimum size=\R]{};
\node at (2.5*\x,0*\x-\sos) {$1$};

\draw (3*\x,0) node[fill,circle,minimum size=\R]{};
\node at (3*\x+\sos,0.5*\x) {$1$};
\draw (3*\x,\x) node[fill,circle,minimum size=\R,color=c1]{};
\node at (2.5*\x,\x+\sos) {$-1$};

\draw (2*\x,2*\x) node[fill,circle,minimum size=\R,color=c1]{};
\draw (3*\x,2*\x) node[fill,circle,minimum size=\R]{};
\node at (3*\x+\sos,1.5*\x) {$1$};
\node at (2*\x-\sos,1.5*\x) {$1$};
\node at (2.5*\x,2*\x+\sos) {$1$};

\draw (2*\x,3*\x) node[fill,circle,minimum size=\R]{};
\node at (2*\x-\sos,2.5*\x) {$1$};
\draw (3*\x,3*\x) node[fill,circle,minimum size=\R,color=c1]{};
\node at (3*\x,2.5*\x) {{$-1$}};

\draw (4*\x,2*\x) node[fill,circle,minimum size=\R,color=c1]{};
\draw (4*\x,3*\x) node[fill,circle,minimum size=\R]{};
\node at (4*\x+\sos,2.5*\x) {{$1$}};

\draw (5*\x,2*\x) node[fill,circle,minimum size=\R]{};
\node at (3.5*\x,2*\x-\sos) {$1$};
\node at (4.5*\x,2*\x-\sos) {$1$};
\node at (5.5*\x,2*\x-\sos) {$1$};
\draw (5*\x,3*\x) node[fill,circle,minimum size=\R,color=c1]{};
\node at (5*\x,2.5*\x) {$-1$};
\node at (6*\x+\sos,2.5*\x) {$1$};

\draw (6*\x,2*\x) node[fill,circle,minimum size=\R,color=c1]{};
\draw (6*\x,3*\x) node[fill,circle,minimum size=\R]{};
\node at (5.5*\x,3*\x+\sos) {{$1$}};
\node at (4.5*\x,3*\x+\sos) {{$1$}};
\node at (3.5*\x,3*\x+\sos) {{$1$}};
\node at (2.5*\x,3*\x+\sos) {{$1$}};

\begin{pgfonlayer}{background}
\path[draw] (0,0) to (3*\x,0);
\path[draw] (0,\x) to (3*\x,\x);

\path[draw] (2*\x,0) to (2*\x,3*\x);
\path[draw] (3*\x,0) to (3*\x,3*\x);

\path[draw] (2*\x,2*\x) to (6*\x,2*\x);
\path[draw] (2*\x,3*\x) to (6*\x,3*\x);

\path[draw] (4*\x,2*\x) to (4*\x,3*\x);
\path[draw] (5*\x,2*\x) to (5*\x,3*\x);
\path[draw] (6*\x,2*\x) to (6*\x,3*\x);

\path[draw] (0,0) to (0,\x);
\path[draw] (\x,0) to (\x,\x);

\end{pgfonlayer}

\end{tikzpicture}

\end{center}

\caption{A real Kasteleyn weighting}
\label{fig:weighting}
\end{figure}

Given a Kasteleyn weighting $w$ of a plane bipartite graph $G=(V,E)$, we can construct another weighting $\widetilde{w}$ of $G$ using the following procedure. We pick a vertex $v\in V$ and change the weighting on all edges incident with $v$, that is, we put
\begin{align}
    \label{Eqn:GaugeTrafo}
    \widetilde{w}(e)=\begin{cases}
    -w(e)&\textrm{if $e$ is incident with $v$;}\\
    w(e)&\textrm{otherwise.}
    \end{cases}
\end{align}
It is easy to see that $\widetilde{w}$ is again a Kasteleyn weighting. Moreover, any two Kasteleyn weightings of $G$ are related to each other by applying this procedure for a suitable sequence of vertices, see Kenyon~\cite[Section 3.3]{K09}.

\begin{prop}
Let $w, \widetilde{w}\colon E\to \{\pm 1\}$ be two Kasteleyn weightings of the same bipartite plane graph $G=(V,E)$, and let $A, \widetilde{A}$ be their associated weighted adjacency matrices. Then $\widetilde{A}$ is unitarily equivalent to $A$ (i.e. $\widetilde{A}=UAU^*=UAU^{-1}$ for some unitary matrix~$U$).
\end{prop}

\begin{proof}
An application of the procedure described in Equation (\ref{Eqn:GaugeTrafo}) to a vertex $v_i$ corresponds to multiplying both row and column $i$ of the associated weighted adjacency matrix by a sign. The full procedure corresponds to a sequence of such operations, and is equivalent to pre- and post-multiplication of the matrix by a single diagonal matrix $D$ with diagonal entries in $\{\pm 1\}$. Hence $\widetilde{A}=DAD$ and clearly $D$ is both unitary and self-adjoint.
\end{proof}

\begin{defn}[Bipartite weighted adjacency matrices]
\label{Def:BipartiteAdjacency}
The submatrix $B$ of $A$ on rows $V_1$ and columns $V_2$ is called the \emph{bipartite adjacency matrix} of $G$. 
\end{defn}

The bipartite weighted adjacency matrix $B$ determines the weighted adjacency matrix $A$. Specifically, we can reorder the vertices in such a way that
\begin{align}
\label{Eqn:Bipartite}
    A=\left(
    \begin{matrix}
    0 & B\\
    B^T & 0
    \end{matrix}
    \right)
\end{align}
with $B\in\operatorname{Mat}_{m_1\times m_2}(\mathbb{Z})$ as in Definition~\ref{Def:BipartiteAdjacency} and zero matrices of size $m_1\times m_1$ and $m_2\times m_2$. Note that the matrix $BB^T$ is symmetric and positive semidefinite. In particular, its eigenvalues are non-negative real numbers.

Some authors, for example Kenyon \cite{K09}, also use the name \emph{Kasteleyn matrix} for the matrix $B$. Moreover, if $m_1=m_2$, then $\operatorname{det}(A)=\operatorname{det}(BB^T)=\operatorname{det}(B)^2$.

\begin{theo}[Kasteleyn \cite{K61,K63}, Temperley--Fisher \cite{TF61}]
\label{Thm:Kasteleyn}
Let $w\colon E\to \{\pm 1\}$ be a Kasteleyn weighting of a plane bipartite graph $G=(V_1\sqcup V_2,E)$ with $\lvert V_1\rvert=\lvert V_2\rvert$, and let $A$ be the associated weigthed adjacency matrix. Then
\begin{align}
\label{Eqn:NumberMatch}
        \left\vert \operatorname{det}(A)\right\rvert=\PM(G)^2\textnormal{ and }\lvert\operatorname{det}(B)\rvert=\PM(G).
    \end{align}
\end{theo}

\begin{proof}[Sketch of the proof] Let us abbreviate $n=m_1=m_2$ and label the black vertices by $\{v_0,\ldots,v_{n-1}\}$ and the white vertices $\{w_0,\ldots,w_{n-1}\}$. We expand the determinant as
\begin{align*}
    \operatorname{det}(B)=\sum_{\pi \in S_n} \operatorname{sgn}(\pi) b_{0,\pi(0)}b_{1,\pi(1)}\cdot\ldots\cdot b_{n-1,\pi(n-1)}.
\end{align*}
By construction the summand associated with a permutation $\pi$ vanishes unless $(v_i,w_{\pi(i)})$ is an edge in $G$ for every $i\in [0,n-1]$. In this case those edges form a perfect matching of $G$. The corresponding summand is equal to $+1$ or $-1$, and the conditions that a Kasteleyn weighting must fulfill (see Definition~\ref{Def:Weighting}) imply that all summands have the same sign.
\end{proof}

\subsection{Weighted adjacency matrices of bipartite graphs}
\label{Subsec:WeightedAdjacency}

Suppose we are given a bipartite graph $G=(V,E)$ with bipartition $V=V_1\sqcup V_2$ of order $(m_1,m_2)$ together with a weighting $w\colon E\to R$. 

\begin{prop}
\label{Prop:Eigenvectors}
Let $t\neq 0$ be real number. A vector $v\in \mathbb{R}^{m_1}$ is an eigenvector of $BB^T$ of eigenvalue $t^2$ if and only if
\begin{align*}
\begin{pmatrix}
tv\\ B^Tv
\end{pmatrix},
\begin{pmatrix}
-tv\\ B^Tv
\end{pmatrix}
\in\mathbb{R}^{m_1+m_2}
\end{align*}
are eigenvectors of $A$ of eigenvalues $t$ and $-t$.
\end{prop}

\begin{proof} Every $v\in \mathbb{R}^{m_1}\backslash \{0\}$ satisfies
\begin{align*}
 \begin{pmatrix}
0&B \\
B^T & 0
\end{pmatrix}\begin{pmatrix}
\pm tv\\ B^Tv
\end{pmatrix}=\begin{pmatrix}
BB^Tv\\ \pm t B^Tv
\end{pmatrix}.
\end{align*}
This vector agrees with the (nonzero) vector
\begin{align*}
\begin{pmatrix}
t^2v\\ \pm t B^Tv
\end{pmatrix}=\pm t\begin{pmatrix}
\pm tv\\ B^Tv
\end{pmatrix}
\end{align*}
if and only if $BB^Tv=t^2 v$, that is, $v$ is an eigenvector of $BB^T$ with eigenvalue $t^2$.
\end{proof}

In particular, Proposition~\ref{Prop:Eigenvectors} implies that if the determinants of $A$ and $BB^T$ are non-zero, then they have the same absolute value, since they can be expressed as products of their eigenvalues.

\begin{prop}
\label{Prop:CharPoly}
Suppose that $\lvert V_1\rvert=\lvert V_2\vert$. The characteristic polynomials of $A$ and $BB^T$ are related to each other by the relation $\charpoly{A}(t)=\charpoly{BB^T}(t^2)$.
\end{prop}

\begin{proof}
A block matrix with $2\times 2$ blocks of size $m_1\times m_1$ satisfies the identity
\begin{align*}
    \operatorname{det}\begin{pmatrix}X_1&X_2\\X_3&X_4\end{pmatrix}=\operatorname{det}\left(X_1-X_2X_4^{-1}X_3\right)\operatorname{det}(X_4),
\end{align*}
provided $X_4$ is invertible. Now suppose that $t$ is a nonzero real number. We plug in $X_1=X_4=t I_{m_1}$, $X_2=-B$, and $X_3=-B^T$. We obtain
\begin{align*}
    \charpoly{A}(t)=\operatorname{det}\left(tI_{m_1}-A\right)&=\operatorname{det}\left(tI_{m_1}-t^{-1}BB^T\right)t^{m_1}=\operatorname{det}\left(t^2 I_{m_1}-BB^T\right)=\charpoly{BB^T}\left(t^2\right).
\end{align*}
Since the identity holds for all $t\in \mathbb{R}\backslash\{0\}$, both polynomials must agree.
\end{proof}

\subsection{On the bipartite structure of snake graphs}
\label{Subsec:BipartiteStructure}

Now we assume that $G=(V,E)$ is a snake graph with a bipartition $V=V_1 \sqcup V_2$ of the vertices into black and white. Furthermore, we equip $G$ with a Kasteleyn weighting $w\colon E\to \{\pm 1\}$. In this case $BB^T$ is a square matrix whose rows and columns are both indexed by the black vertices of $G$. By construction, there are as many black as white vertices.

\begin{rem}
Every snake graph admits at least one perfect matching. So the bipartite matrix $B$ is invertible by Equation~(\ref{Eqn:NumberMatch}). It follows that $\operatorname{det}(BB^T)=\operatorname{det}(B)^2\neq 0$. Hence, all eigenvalues of $BB^T$ are positive and $BB^T$ is positive definite.
\end{rem}

\begin{prop}
\label{Prop:DiagonalEntries}
Given a black vertex $i\in V_1$. The diagonal entry of $BB^T$ at position $i$ is equal to the number of white vertices that are adjacent to $i$.
\end{prop}

\begin{proof}
For all white vertices $k \in V_2$, we have
\begin{align*}
b_{ik}^2=
\begin{cases}
1 &\text{if there is an edge from $k$ to $i$;} \\
0 &\text{otherwise}.
\end{cases}
\end{align*}
Hence, the diagonal entry is given by
\begin{align*}
    \left(BB^T\right)_{ii}=\sum_{k \in V_2} b_{ik} b^T_{ki}=\sum_{k \in V_2} b_{ik}^2
     = \lvert \{ k \in V_2 \setst{} \text{$k$ is adjacent to $i$}\} \rvert
\end{align*}
by the nature of matrix multiplication.
\end{proof}

Recall that turns are vertices in a snake graph with $2$ or $4$ neighbours that do not lie at the beginning or the end, see Definition~\ref{Def:Turns}. We are particularly interested in snake graph for which all turns have the same colour. Equivalently, except at the beginning and the end, the length of any maximal sequence of consecutive tiles propagating in the same direction must be odd. For example, the graph in Figure~\ref{fig:weighting} satisfies this property.

Given a graph with this property, we assume that the turns are coloured black without loss of generality.

\begin{prop}
\label{Prop:NoInteraction}
Assume that $G$ is a snake graph such that all turns are black. Let $i$ be a black vertex on the upper boundary of $G$ and let $j$ be a black vertex on the lower boundary of $G$. Then $(BB^T)_{ij}=0$.
\end{prop}

\begin{proof}
First, suppose that $i$ and $j$ are opposite vertices of a square tile. Then there are exactly two white vertices $k$ and $l$ which are adjacent to both $i$ and $j$. We have $b_{ik}b_{jk}b_{il}b_{jl}=-1$, since the weighting $w$ is assumed to be a Kasteleyn weighting. We conclude that
\begin{align*}
    \left(BB^T\right)_{ij} = b_{ik}b_{jk}+b_{il}b_{jl}=0.
\end{align*}
Second, if $i$ and $j$ are not opposite vertices of a square tile, then $(BB^T)_{ij}=0$, since $i$ and $j$ do not have any common neighbours.
\end{proof}

\begin{prop}
\label{Prop:Tridiagonal}
Assume that $G$ is a snake graph such that all turns are black. We endow the graph with the Kasteleyn weighting from Definition~\ref{Def:ConstructionRealWeighting} Then we can reorder the black vertices $V$ such that $BB^T$ becomes a $2\times 2$ block matrix
\begin{align*}
   BB^T= \begin{pmatrix}
   B_1 & 0\\
   0& B_2
   \end{pmatrix}
\end{align*}
with blocks indexed by the upper and lower boundary. The matrices $B_1$ and $B_2$ are tridiagonal matrices %Moreover, the Kasteleyn weighting $w$ can be chosen
such that the entries on the first diagonal above and below the main diagonal are all $1$.
\end{prop}

\begin{proof} Proposition~\ref{Prop:NoInteraction} implies that we can reorder the black vertices so that $BB^T$ is a $2\times 2$ block matrix as above, with blocks indexed by the black vertices in the upper and lower boundary. The upper and the lower boundary are path graphs with vertices alternating in colour. By construction of the weighting we have $w(e)=1$ for all edges $e$ in the upper and lower boundary. 

Let $i$ and $j$ be black vertices in the same boundary component such that the only white vertex in $G$ adjacent to both of them, say $k$, lies in the same boundary component between them. If $e_1=(i,k)$ and $e_2=(j,k)$, then $(BB^T)_{ij}=w(e_1)w(e_2)=1\cdot 1=1$.
\end{proof}

For the particular case of horizontal snake graphs, the characteristic polynomials $\charpoly{A}(t)=\charpoly{BB^T}(t^2)$, as well as the eigenvalues and eigenvectors, have been determined independently by Kasteleyn~\cite{K61,K63} and Temperley--Fisher~\cite{TF61}.

\begin{theo}[Kasteleyn, Temperley--Fisher]
\label{Thm:Horizontal}
Suppose that $G=H_n$. The eigenvalues of $BB^T$ are given by
\begin{align*}
    4\cos^2 \left(\frac{l \pi}{n+2}\right)+1&&(l=1,\ldots,n+1).
    \end{align*}
\end{theo}

Notice that in Theorem~\ref{Thm:Horizontal} the indices $l$ and $n+1-l$ yield the same eigenvalue. Especially, all eigenvalues occur with multiplicity $2$, except possibly for the eigenvalue $1$, which occurs with multiplicity $1$ in the case when $n$ is even for the choice $l=(n+2)/2$.

\begin{coro}
\label{Coro:FibonacciProd}
Combining Example~\ref{Ex:Fibonacci} and Theorem~\ref{Thm:Horizontal} we obtain
\begin{align*}
    F_{n+2}=\prod_{l=1}^{\lfloor \frac{n+1}{2}\rfloor} \left(4\cos^2 \left(\frac{l \pi}{n+2}\right)+1\right).
\end{align*}
\end{coro}

The polynomials having the factors of the product on the right hand side in Corollary~\ref{Coro:FibonacciProd} as roots are called \emph{Fibonacci product polynomials}. They will be studied in Section \ref{Sec:FibProd}.

\subsection{Equivalence of Kasteleyn weightings and Pfaffian orientations}
There is an alternative version of Kasteleyn's Theorem using Pfaffians of skew-symmetric matrices instead of determinants of symmetric matrices. In this subsection we will see that both approaches are equivalent.

\begin{defn}[Pfaffian orientations]
Given the weighting $w\colon E\to \{\pm 1\}$ from Definition~\ref{Def:ConstructionRealWeighting} of $G$, we can construct an orientation as follows.
\begin{itemize}
    \item[(i)] An edge $e$ with $w(e)=1$ is oriented from black to white.
    \item[(ii)] An edge $e$ with $w(e)=-1$ is oriented from white to black.
\end{itemize}
\end{defn}

\begin{rem}
It follows immediately that the orientation is \emph{Pfaffian}, see Aigner \cite[Section~10.1]{A07}.
\end{rem}

The \emph{oriented adjacency matrix} $\widetilde{A}$ is a square matrix of size $m=\lvert V\rvert$ with entries in $\{\pm 1\}$ such that $\widetilde{A}_{ij}=-\widetilde{A}_{ji}=1$ for every arrow $e\colon i\to j$ in the oriented graph and $\widetilde{A}_{ij}=0$ if there is no arrow between $i$ and $j$. By construction $\widetilde{A}$ is skew-symmetric, and if we reorder the vertices as in Equation~(\ref{Eqn:Bipartite}), then
\begin{align*}
    \widetilde{A}= \begin{pmatrix}
0&B \\
-B^T & 0
\end{pmatrix}\in\operatorname{Mat}_{m\times m}(\mathbb{Z}).
\end{align*}
The determinant of the $m\times m$ skew-symmetric matrix $\widetilde{A}$ is the square of an integer polynomial in the entries called the \emph{Pfaffian} $\operatorname{Pf}(\widetilde{A})$. In short, $\operatorname{Pf}(\widetilde{A})^2=\operatorname{det}(\widetilde{A})$. In particular, Equation~(\ref{Eqn:NumberMatch}) implies that
\begin{align*}
\left\lvert \operatorname{Pf}(\widetilde{A})\right\rvert = \PM(G)
\end{align*}
is equal to the number of perfect matchings of $G$.

The next statement is a variation of Proposition~\ref{Prop:Eigenvectors} and the proof carries over mutatis mutandis to our setting.

\begin{prop}
Let $t\neq 0$ be real number. A vector $v\in \mathbb{R}^{m_1}$ is an eigenvector of $BB^T$ of eigenvalue $t^2$ if and only if
\begin{align*}
\begin{pmatrix}
i tv\\ B^Tv
\end{pmatrix},
\begin{pmatrix}
-i tv\\ B^Tv
\end{pmatrix}
\in\mathbb{R}^{m_1+m_2}
\end{align*}
are eigenvectors of $\widetilde{A}$ of eigenvalues $it$ and $-it$.
\end{prop}

We conclude that the knowledge about spectral properties of the matrix $BB^T$ is sufficient to determine spectral properties of the skew-symmetric matrix $\widetilde{A}$.

\section{Chebyshev polynomials and orthogonal polynomials}

In this section we recall some basic facts about Chebyshev polynomials and orthogonal polynomials. Both notions will become crucial when studying the characteristic polynomials we have encountered in Section~\ref{Sec:Recursions}.

\subsection{Chebyshev polynomials}
\label{Sec:Chebyshev}

\begin{defn}[Chebyshev polynomials of the first kind] The sequence $(T_n)_{n\geq 0}$ of polynomials in $\mathbb{Z}[y]$ is defined by $T_0(y)=1$, $T_1(y)=y$ and $T_{n+1}(y)=2yT_n(y)-T_{n-1}(y)$.
\end{defn}

\begin{defn}[Chebyshev polynomials of the second kind] The sequence $(U_n)_{n\geq 0}$ of polynomials in $\mathbb{Z}[y]$ is defined by $U_0(y)=1$, $U_1(y)=2y$ and $U_{n+1}(y)=2yU_n(y)-U_{n-1}(y)$.
\end{defn}

The recursion for $(U_n)$ is consistent with the convention $U_{-1}(y)=0$. Notice that $T_n(y)$ and $U_n(y)$ are polynomials of degree $n$ for every $n\geq 1$. The leading coefficient of $T_n(y)$ is $2^{n-1}$, the leading coefficient of $U_n(y)$ is $2^n$.

In the article we will need some identities for Chebyshev polynomials. Before we state those identities, let us recall some trigonometric identities on which our calculations are built.

\begin{rem}[Product to sum formula]
\label{Rem:TrigIdentities}
Let $\theta$ and $\eta$ be real numbers. Then
\begin{align*}
    \textrm{(i)} \quad 2\cos(\theta)\cos(\eta)&=\cos(\theta-\eta)+\cos(\theta+\eta),\\
    \textrm{(ii)} \quad 2\sin(\theta)\cos(\eta)&=\sin(\theta-\eta)+\sin(\theta+\eta),\\
    \textrm{(iii)} \quad 2\sin(\theta)\sin(\eta)&=\cos(\theta-\eta)-\cos(\theta+\eta).
\end{align*}
\end{rem}
Especially, if we plug in $\theta=\eta$ in the first and second equation of the previous remark, then we obtain the double angle formulae $\cos(2\theta)=2\cos^2(\theta)-1=1-2\sin^2(\theta)$ and $\sin(2\theta)=2\cos(\theta)\sin(\theta)$.

Chebyshev polynomials arise in mathematics in many contexts and in particular, they are related to trigonometry. Equations~(i) and (ii) in Remark~\ref{Rem:TrigIdentities} imply that, for any $\theta\in\mathbb{R}$ and $n\in \mathbb{N}$, we may write
\begin{align}
    \label{Eqn:ChebyshevSine}
      T_{n}\left(\cos(\theta)\right)=\cos\left(n\theta\right),&&  U_{n}\left(\cos(\theta)\right)=\frac{\sin\left((n+1)\theta\right)}{\sin\left(\theta\right)}.
\end{align}

\begin{rem}
\label{Rem:RootsChebyshev}
Equation~(\ref{Eqn:ChebyshevSine}) allows us read off the roots of the Chebyshev polynomials of the first and second kind. Specifically, $T_n(y)$ vanishes for
\begin{align*}
    y_l=\cos\left(\frac{\pi(2l+1)}{2n}\right)&&(l=0,1,\ldots,n-1);
\end{align*}
since $\operatorname{deg}(T_n)=n$, these are all the roots of $T_n$. Similarly, the roots of $U_n$ are
\begin{align*}
    y_l=\cos\left(\frac{\pi l}{n+1}\right)&&(l=1,2,\ldots,n).
\end{align*}
\end{rem}

\begin{prop}
\label{Prop:IdentitiesChebyshev}
For every $n\geq 0$ we have
\begin{align*}
    \textnormal{(i)} \quad U_{2n+1}(y)=2U_n(y)T_{n+1}\left(y\right), &&
    \textnormal{(ii)} \quad U_{2n}(y)=U_n(y)^2-U_{n-1}(y)^2.
\end{align*}
\end{prop}

\begin{proof}
The first statement follows from the double angle identity for $\sin((2n+2)\theta)$.

To prove the second statement, we use Equation (\ref{Eqn:ChebyshevSine}), the identity (iii) from Remark~\ref{Rem:TrigIdentities} and the double angle formula for cosine. We get, for all $\theta \in \left(0,\pi\right)$, that
\begin{align*}
    2\sin^2(\theta)U_{2n}(\cos(\theta))
    &= 2\sin(\theta)\sin\left((2n+1)\theta\right) \\
    &= \cos(2n\theta)-\cos\left((2n+2)\theta\right) \\
    &= \left[ 1-2\sin^2(n\theta) \right] -\left[ 1-2\sin^2\left((n+1)\theta\right) \right] \\
    &= 2\sin^2\left((n+1)\theta\right)-2\sin^2(n\theta) \\
    &= 2\sin^2(\theta)U_n(\cos(\theta))^2 - 2\sin^2(\theta)U_{n-1}(\cos(\theta))^2.
\end{align*}
This clearly proves that the statement holds for all $y \in (-1,1)$.
This is enough to show that the statement holds in $\mathbb{Z}[y]$, since the polynomials on both sides agree for infinitely many real numbers.
\end{proof}

\subsection{Orthogonal polynomial sequences}
\newcommand{\orthP}{P}

%In this section we recall some well-known facts about orthogonal polynomials which will become helpful in later sections of the article.

\begin{defn}[Orthogonal polynomial sequence]
\label{Def:OPSrec}
Given a monic polynomial sequence $(\orthP_n)_{n\geq 0}$ with $\operatorname{deg}(\orthP_n)=n$ and real coefficients, we say that $(\orthP_n)$ is an \emph{orthogonal polynomial sequence} (OPS) if and only if it satisfies a recurrence of the form
\begin{gather*}
    \orthP_{n+1}(x) = (x-\beta_n)\orthP_n(x)-\gamma_n \orthP_{n-1}(x),\ n\geq 1 \\
    \orthP_0(x)=1,\ \orthP_1(x)=x-\beta_0
\end{gather*}
for some real coefficients $(\beta_n)_{n\geq 0}$ and $(\gamma_n)_{n\geq 1}$ with $\gamma_n\neq 0$ for $n\geq 1$.

Given instead a polynomial sequence $(\widetilde{P}_n)_{n\geq 0}$ with $\operatorname{deg}(\widetilde{P}_n)=n$ and real coefficients, we say that $(\widetilde{P}_n)$ is an OPS if and only if the corresponding monic polynomial sequence $(\orthP_n)$, given by dividing each polynomial by its leading coefficient, satisfies a recurrence relation of the form given above.
\end{defn}

\begin{exam}
The sequences $(T_n)_{n\geq 0}$ and $(U_n)_{n\geq 0}$ of Chebyshev polynomials are OPS.
\end{exam}

Other equivalent definitions to Definition~\ref{Def:OPSrec} exist involving linear functionals, in particular the polynomials in an OPS are pairwise orthogonal with respect to a suitable inner product.

\begin{theo}[Christoffel--Darboux]
Given $(\orthP_n)_{n\geq 0}$ a monic OPS with recurrence as in Definition~\ref{Def:OPSrec}, we have the following identity
\[ \frac{\orthP_{n+1}(x)\orthP_n(y)-\orthP_n(x)\orthP_{n+1}(y)}{x-y} =\sum_{k=0}^n\gamma_{k+1}...\gamma_n \orthP_k(x)\orthP_k(y) \]
Taking $\lim_{y\to x}$ gives the \emph{confluent Christoffel--Darboux identity}
\[ P'_{n+1}(x)\orthP_n(x)-\orthP_{n+1}(x)P'_n(x) =\sum_{k=0}^n\gamma_{k+1}...\gamma_n(\orthP_k(x))^2. \]
\end{theo}

\begin{rem}
\label{Rem:OPS-CDinequality}
In particular, if $\gamma_n>0$ for all $n\geq 1$ then
\[ P'_{n+1}(x)\orthP_n(x) - \orthP_{n+1}(x)P'_n(x) > 0 \]
\end{rem}

\begin{theo}
\label{Thm:OPSsimple}
The real zeros of any $\orthP_n$ in a monic OPS $(\orthP_n)_{n\geq 0}$ with recurrence as in Definition~\ref{Def:OPSrec} and $\gamma_n>0$ for all $n\geq 1$ are all simple.
\end{theo}

\begin{proof}
Let $x_0$ be a real root of $\orthP_{n+1}$, then with $x=x_0$ Remark~\ref{Rem:OPS-CDinequality} becomes
\[P'_{n+1}(x_0)\orthP_n(x_0) > 0,\]
therefore $P'_{n+1}(x_0)\neq 0$, which proves the theorem.
\end{proof}

\begin{rem}
A more complicated proof involving linear functionals can show that any OPS as in Theorem~\ref{Thm:OPSsimple} has only real zeros. This is not included here.
\end{rem}

\section{Tridiagonal matrices and characteristic polynomials}
\label{Sec:Recursions}

\subsection{Tridiagonal matrices}

Each of the matrices $B_1$ and $B_2$ in Proposition~\ref{Prop:Tridiagonal} can be written as
\begin{align}
\label{Eqn:Tridiagonal}
\begin{pmatrix}
e_0 & 1 \\
1   & e_1 & 1 \\
    & 1   & e_2 & 1\\
    &      &\ddots&\ddots & \ddots \\
    & & & 1 & e_{n-2} & 1\\
    &    & &  & 1 & e_{n-1}
\end{pmatrix}
\end{align}
for some integers $e_0,\ldots,e_{n-1}$ (which will be different for $B_1$ or $B_2$ in general). By Proposition~\ref{Prop:DiagonalEntries}, the number $e_i$ is the number of white vertices that are adjacent to a black vertex $i$. In this section we wish to study characteristic polynomials of tridiagonal matrices as in Equation~(\ref{Eqn:Tridiagonal}). The knowledge of the eigenvectors and eigenvalues makes it possible to determine the number of perfect matchings of the graph as a product of the eigenvalues thanks to Kasteleyn's Theorem~\ref{Thm:Kasteleyn}.

Let us fix an infinite sequence of integers $e=(e_k)_{k\geq0}$.

\begin{defn}[Tridiagonal matrices]
\label{Def:Tridiagonal}
For a natural number $n$, we denote the tridiagonal $n\times n$ matrix with sub- and super-diagonal entries $(1,1,\ldots,1)$ and diagonal entries $(e_0,e_1,\ldots,e_{n-1})$ in Formula~(\ref{Eqn:Tridiagonal}) by $M_n$. We adopt the convention that $M_0$ is the empty matrix.
\end{defn}

 In the next lemma we use the convention $\charpoly{M_0}(x)=1$.

\begin{lemma}
\label{Lemma:Tridiagonal}
For every $n\geq 1$ we have
\begin{align}
\label{Eqn:Recursion}
\charpoly{M_{n+1}}(x)=(x-e_n)\charpoly{M_n}(x)-\charpoly{M_{n-1}}(x).
\end{align}
\end{lemma}

Lemma~\ref{Lemma:Tridiagonal} is well-known and can be proved by expanding $\charpoly{M_{n+1}}(x)=\operatorname{det}(x I_{n+1}-M_{n+1})$ along the last column. It provides a recursion to compute the characteristic polynomials. Notice that Equation (\ref{Eqn:Recursion}) is coherent with the choice $\charpoly{M_{-1}}(x)=0$.

\begingroup
\newcommand{\xzero}{x_0}
\newcommand{\vzero}{v}
\newcommand{\vecdots}[5][\vdots]{
#2\\
#3\\
#1\\
#4\\
#5
}
\begin{prop}
\label{Prop:EntriesEigenvectors}
For any zero $\xzero$ of $\charpoly{M_n}(x)$ (i.e. $\xzero$ any eigenvalue of $M_n$), the vector
\[\vzero=\begin{pmatrix}
    \vecdots{\charpoly{M_0}(\xzero)}{\charpoly{M_1}(\xzero)}{\charpoly{M_{n-2}}(\xzero)}{\charpoly{M_{n-1}}(\xzero)}
\end{pmatrix}
\]
is an eigenvector of $M_n$ with eigenvalue $\xzero$ and $\ker(\xzero I_n-M_n)=\operatorname{span}\{\vzero\}$ is the eigenspace of $M_n$ for this eigenvalue.
\end{prop}

\begin{proof}
\newcommand{\charsum}[5][&]{\charpoly{M_{#3}}(#2)#1+#1e_{#4}\charpoly{M_{#4}}(#2)#1+#1\charpoly{M_{#5}}(#2)}
\newcommand{\charsuma}[4][&]{\charpoly{M_{#3}}(#2)#1+#1e_{#4}\charpoly{M_{#4}}(#2)}
\newcommand{\charsumc}[4][&]{e_{#3}\charpoly{M_{#3}}(#2)#1+#1\charpoly{M_{#4}}(#2)}
We can rearrange the recurrence from Lemma~\ref{Lemma:Tridiagonal} to get for every $k\geq 0$
\[x\charpoly{M_k}(x)=\charsum[]{x}{k-1}{k}{k+1}\]
with the choice $\charpoly{M_{-1}}(x)=0$. We then have
\begingroup \setlength\arraycolsep{1pt}
\begin{align*}
M_n
\begin{pmatrix}
    \vecdots{\charpoly{M_0}(\xzero)}{\charpoly{M_1}(\xzero)}{\charpoly{M_{n-2}}(\xzero)}{\charpoly{M_{n-1}}(\xzero)}
\end{pmatrix}
&= \begin{pmatrix}
    0&+&\charsumc{\xzero}{0}{1}\\
    \charsum{\xzero}{0}{1}{2}\\
    &&\vdots\\
    \charsum{\xzero}{n-3}{n-2}{n-1}\\
    \charsuma{\xzero}{n-2}{n-1}
\end{pmatrix}\\
&= \begin{pmatrix}
    \charsum{\xzero}{-1}{0}{1}\\
    \charsum{\xzero}{0}{1}{2}\\
    &&\vdots\\
    \charsum{\xzero}{n-3}{n-2}{n-1}\\
    \charsum{\xzero}{n-2}{n-1}{n}
\end{pmatrix}
- \charpoly{M_n}(\xzero) \begin{pmatrix} \vecdots{0}{0}{0}{1} \end{pmatrix} \\
&=\xzero\vzero-0 = \xzero\vzero
\end{align*}
\endgroup
and $\vzero$ is non-zero since $\charpoly{M_0}(x)=1$, hence $v$ is an eigenvector of $M_n$ with eigenvalue $\xzero$.

Finally the eigenspace $\ker(\xzero I_n-M_n)$ can only have dimension 1 since all roots
of $\charpoly{M_n}(x)$ are simple by Theorem~\ref{Thm:OPSsimple}.
\end{proof}
\endgroup

\subsection{Characteristic polynomials via Chebyshev polynomials}

In this subsection we relate the characteristic polynomials describing entries of eigenvectors to the Chebyshev polynomials of the second kind from Section~\ref{Sec:Chebyshev}. Recall that the sequence $(U_n)$ satisfies the recursion $U_{n+1}=2yU_n-U_{n-1}$ for $n\geq 0$ and has initial values $U_{-1}=0$ and $U_0=1$.

Let $\mu$ be a parameter.

\begin{defn}[Shifted characteristic polynomials] We put
$f_n(y)=\charpoly{M_n}(2y+\mu)$.
\end{defn}

\begin{rem}
\label{Rem:TridiagonalShifted}
We can rewrite Lemma~\ref{Lemma:Tridiagonal} as
\begin{align*}
f_{n+1}(y)&=(2y+\mu-e_n)f_n(y)-f_{n-1}(y)\\
          &=[2yf_n(y)-f_{n-1}(y)]+(\mu-e_n)f_n(y).
\end{align*}
\end{rem}

\begin{lemma}
\label{Lemma:TridiagCheb}
For every $n \geq 0$ we have
\[
f_n=U_n+\sum_{k=0}^{n-1}(\mu-e_k)f_k U_{n-k-1}.
\]
\end{lemma}

\begin{proof}
We prove the lemma by induction on $n$. Notice that, by definition $f_{0}=1$ and $f_1=2y+(\mu-e_0)$, so that the statement is true for $n=0$ and $n=1$. For the induction step, we assume that this formula holds for both
$f_n$ and $f_{n-1}$. Using $U_{n-(n-1)-2}=U_{-1}=0$ we may rephrase the induction hypothesis as
\begin{align*}
f_{n-1}&=U_{n-1}+\sum_{k=0}^{n-2}(\mu-e_k) f_k U_{n-k-2}
       =U_{n-1}+\sum_{k=0}^{n-1}(\mu-e_k) f_k U_{n-k-2} \\
f_n   &=U_n+\sum_{k=0}^{n-1}(\mu-e_k) f_k U_{n-k-1}
\end{align*}
which can be used with Remark~\ref{Rem:TridiagonalShifted} to get
\begin{align*}
f_{n+1}&=\left[2yf_n-f_{n-1}\right]+(\mu-e_n)f_n \\
       &=\left[2yU_n + 2y\sum_{k=0}^{n-1}(\mu-e_k) f_k U_{n-k-1} - U_{n-1}-
         \sum_{k=0}^{n-1}(\mu-e_k) f_k U_{n-k-2}\right] + (\mu-e_n)f_n U_{n-n} \\
       &=\left[2yU_n - U_{n-1} +
         \sum_{k=0}^{n-1}(\mu-e_k)f_k (2yU_{n-k-1} - U_{n-k-2})\right]
         + (\mu-e_n)f_n U_{n-n} \\
       &=U_{n+1} + \sum_{k=0}^{n-1}(\mu-e_k)f_k U_{n-k} + (\mu-e_n)f_n U_{n-n}.\qedhere % \\
       %&=U_{n+1} + \sum_{k=0}^{n}(\mu-e_k)f_k U_{n-k}.\qedhere
\end{align*}
\end{proof}

\subsection{Fibonacci product polynomials}
\label{Sec:FibProd}

\newcommand{\fibP}{P}
\newcommand{\fibQ}{Q}

\begin{defn}[Fibonacci product polynomials]
\label{Def:FibProdPoly}
The \emph{Fibonacci product polynomials} are two monic polynomial sequences $(\fibP_n)_{n\geq 0}$ and $(\fibQ_n)_{n\geq 0}$ defined in terms of characteristic polynomials. For any $n\geq 1$,
\begin{itemize}
    \item[(i)] $\fibP_n(x)$ is the characteristic polynomials of the $n\times n$ matrix
    \begin{align}
\label{Eqn:FibP_def}
\begin{pmatrix} \trimat{1}{2}{3}{3}{3} \end{pmatrix};
\end{align}
\item[(ii)] $\fibQ_n(x)$ is the characteristic polynomial of the $n\times n$ matrix
\begin{align}
\label{Eqn:FibQ_def}\begin{pmatrix} \trimat{1}{3}{3}{3}{3} \end{pmatrix}.
\end{align}
\end{itemize}
We put $\fibP_0(x)=\fibQ_0(x)=1$.
\end{defn}

The polynomials in Definition~\ref{Def:FibProdPoly} are well-known in the literature. Information about them can be found for example in the On-Line Encyclopedia of Integers Sequences, see entry \href{https://oeis.org/A152063}{A152063} \cite{OEIS}. The name reflects a property that we will see later in Proposition~\ref{Prop:RootsFibonacciProdPoly}, namely that the product of the roots of any of these polynomials is a Fibonacci number. This property is related to Corollary~\ref{Coro:FibonacciProd}.

Notice that some authors define the Fibonacci product polynomials using tridiagonal matrices with entries $-1$ on the sub- and superdiagonal. Both definitions are equivalent by the next Lemma~\ref{Lemma:TridiagonalSimilar}; its proof is straightforward.

\begin{lemma}
\label{Lemma:TridiagonalSimilar}
\newcommand{\tmpd}{D}
\newcommand{\tmpm}{\widetilde{M}_n}
With $n\times n$ matrices
\[
\tmpd = \begin{pmatrix} \diamat{1}{-1}{1}{-1} \end{pmatrix},
\tmpm = \begin{pmatrix} \trimat{-1}{e_0}{e_1}{e_2}{e_3} \end{pmatrix}
\]
we have, with $M_n$ as in Lemma \ref{Lemma:Tridiagonal}, that
\[
\tmpd\tmpm\tmpd^{-1} = \begin{pmatrix} \trimat{1}{e_0}{e_1}{e_2}{e_3} \end{pmatrix} = M_n.
\]
So $\tmpm$ and $M_n$ are similar, especially $\charpoly{\tmpm}(x)=\charpoly{M_n}(x)$, meaning Lemma~\ref{Lemma:Tridiagonal} and the results that follow can be used for $\charpoly{\tmpm}(x)$. 
\end{lemma}

\begin{rem}
\label{Rem:FibProdRec}
Each matrix in Definition~\ref{Def:FibProdPoly} is equal to an $M_n$ from Definition~\ref{Def:Tridiagonal}. Specifically, we have $e_n=3$ for all $n\geq 1$ and $e_0=2$ (in the case of Equation~(\ref{Eqn:FibP_def})) or $e_0=3$ (in the case of Equation~(\ref{Eqn:FibQ_def})) respectively. The recursion in Lemma~\ref{Lemma:Tridiagonal} implies that, for $n\geq 1$,
\begin{align*}
\fibP_{n+1}(x) &= (x-3)\fibP_n(x) - \fibP_{n-1}(x), \\
\fibQ_{n+1}(x) &= (x-3)\fibQ_n(x) - \fibQ_{n-1}(x)
\end{align*}
with $\fibP_0(x)=\fibQ_0(x)=1$ and $\fibP_1(x)=x-2$, $\fibQ_1(x)=x-3$.
\end{rem}

\begin{prop}
\label{Prop:FibProdCheb}
Let $n\geq 0$ be a natural number. Then
\begin{align*}
\fibP_n(2y+3) &= U_n(y) + U_{n-1}(y), \\
\fibQ_n(2y+3) &= U_n(y).
\end{align*}
\end{prop}

\begin{proof}
We use Lemma~\ref{Lemma:TridiagCheb}, with parameter $\mu=3$, on the matrices in Definition~\ref{Def:FibProdPoly} to get
\begin{align*}
\fibP_n(2y+3) &= U_n(y) + \sum_{k=1}^{n-1}(3-3)\fibP_k(2y+3)U_{n-k-1}(y) + (3-2)\fibP_0(2y+3)U_{n-1}(y), \\
\fibQ_n(2y+3) &= U_n(y) + \sum_{k=0}^{n-1}(3-3)\fibQ_k(2y+3)U_{n-k-1}(y).
\end{align*}
We get the required identities by removing zero terms and using $\fibP_0(2y+3)=1$. 
\end{proof}

\begin{lemma}
\label{Lemma:RelationPandQ}
For every $n\geq 0$ we have
\begin{align*}
     \fibP_n&=\fibQ_n+\fibQ_{n-1}, \\
(x-1)\fibQ_n&=\fibP_{n+1}+\fibP_n.
\end{align*}
\end{lemma}

\begin{proof}
 The first statement follows immediately from the identities in Proposition~\ref{Prop:FibProdCheb}.

 For the second statement, we can rearrange the recurrence in Remark~\ref{Rem:FibProdRec} to get
 \begin{align*}
     \fibQ_{n+1}(x)+2\fibQ_n(x)+\fibQ_{n-1}(x)=(x-1)\fibQ_n(x)
 \end{align*}
 and we can use the first statement to also get
 \begin{equation*}
     (\fibQ_{n+1}(x)+\fibQ_n(x))+(\fibQ_n(x)+\fibQ_{n-1}(x))=\fibP_{n+1}(x)+\fibP_n(x). \qedhere
 \end{equation*}
\end{proof}

\begin{lemma}
\label{Lemma:TridiagFib}
For every $n\geq 0$ we have
\[
\charpoly{M_n}=\fibQ_n+\sum_{k=0}^{n-1}(3-e_k)\charpoly{M_k} \fibQ_{n-k-1}.
\]
\end{lemma}

\begin{proof}
In Lemma~\ref{Lemma:TridiagCheb} we may replace each $U$ with a $\fibQ$ as given in Proposition~\ref{Prop:FibProdCheb} then substitute $\mu=3$ to get the expression.
\end{proof}

\subsection{Roots of Fibonacci product polynomials} In this subsection we determine the spectrum of the matrices in Equations~(\ref{Eqn:FibP_def}) and (\ref{Eqn:FibQ_def}); in other words, we determine the roots of the Fibonacci product polynomials.

\begin{prop}[Folklore]
\label{Prop:RootsFibonacciProdPoly}
Let $n\geq 0$.
\begin{itemize}
\item[(i)] The roots of the Fibonacci product polynomial $\fibP_n$ are
\begin{align*}
    x_l=4\cos^2\left(\frac{\pi l}{2n+1}\right)+1&&(l=1,2,\ldots,n).
\end{align*}
    \item[(ii)] The roots of the Fibonacci product polynomial $\fibQ_n$ are
\begin{align*}
    x_l=4\cos^2\left(\frac{\pi l}{2(n+1)}\right)+1&&(l=1,2,\ldots,n).
\end{align*}
\end{itemize}
\end{prop}

\newcommand{\tmpargP}[1]{\frac{#1 \pi l}{2n+1}}
\begin{proof} In both parts we use the identity $2\cos(\theta)+3=4\cos^2(\theta/2)+1$ for all $\theta\in\mathbb{R}$ which follows from the double angle formula.

(i) For every $\theta$ we have
\begin{align*}
    \fibP_n(2\cosB{\theta}+3)\sinB{\theta}
    &=\left[U_n(\cosB{\theta})+U_{n-1}(\cosB{\theta})\right]\sinB{\theta}
    =\sinB{(n+1)\theta}+\sinB{n\theta}.
\end{align*}
Let $l\in[1,n]$. To show that $x_l$ is a root of $\fibP_n$, we substitute $\theta=\tmpargP{2}$, for which $\sin(\theta)\neq 0$, in the previous equation and obtain
\begin{align*}
    \fibP_n\left(2\cosB{\tmpargP{2}}+3\right)\sinB{\tmpargP{2}}
    &=\sinB{\tmpargP{2(n+1)}}+\sinB{\tmpargP{2n}}\\
    &=\sinB{\pi l+\tmpargP{}}+\sinB{\pi l-\tmpargP{}}=0.
\end{align*}
Notice that the $x_l$ with $l\in[1,n]$ are pairwise distinct. There cannot exist other roots since $\operatorname{deg}(\fibP_n)=n$.

(ii) We use the relation $\fibQ_n(2y+3)=U_n(y)$. By Remark~\ref{Rem:RootsChebyshev}
the real numbers $y_l=\cosB{\frac{\pi l}{n+1}}$ with $l=1,2,\ldots,n$ are the roots of $U_n(y)$. The double angle implies $2y_l+3=4\cosB[^2]{\frac{\pi l}{2n+2}}+1$.
\end{proof}

\begin{rem}
The determinant of a matrix is equal to the product of the eigenvalues. Combining Proposition~\ref{Prop:RootsFibonacciProdPoly} with Corollary~\ref{Coro:FibonacciProd} we see that the determinant of the matrix (\ref{Eqn:FibP_def}) is equal to $F_{2n+1}$ and that the determinant of the matrix (\ref{Eqn:FibQ_def}) is equal to $F_{2n+2}$.
\end{rem}

\subsection{Divisibility properties of Fibonacci product polynomials}

\begin{prop}
\label{Prop:Divisible}
We have $\fibQ_{r}\mid \fibQ_{2r+1}$ and $\fibP_{r}\mid \fibQ_{2r}$ in $\mathbb{Z}[x]$ for all integers $r\geq 0$.
\end{prop}

\begin{proof}
 We plug $y=(x-3)/2$ in Proposition~\ref{Prop:IdentitiesChebyshev}. The first equation yields
 \begin{align*}
     \fibQ_{2r+1}(x)=2\fibQ_r(x)T_{r+1}\left(\tfrac{x-3}{2}\right).
 \end{align*}
 Using its recursive definition it is easy to see that $2T_r(y)$ is a polynomial in $x$ with integer coefficients when we plug in $y=(x-3)/2$.

 The second equation in Proposition~\ref{Prop:IdentitiesChebyshev} yields
 \begin{align*}
     \fibQ_{2r}(x)=\fibQ_{r}(x)^2-\fibQ_{r-1}(x)^2=\left[\fibQ_{r}(x)-\fibQ_{r-1}(x)\right]\left[\fibQ_{r}(x)+\fibQ_{r-1}(x)\right];
 \end{align*}
 this implies the claim because $\fibQ_{r}(x)+\fibQ_{r-1}(x)=\fibP_r(x)$.
\end{proof}

 \section{Applications}

\subsection{Recursions for graphs with black turns}
\label{Subsec:BlackTurn}

Suppose that $G$ is a snake graph coloured such that all turns are black. By Proposition~\ref{Prop:Tridiagonal} the matrix $BB^T$ is a $2\times 2$ block diagonal matrix whose diagonal blocks $B_1$ and $B_2$ are tridiagonal matrices. The matrix $B_1$ is parametrised by the upper boundary and $B_2$ is parametrised by the lower boundary. We denote the sizes of the blocks by $m_1$ and $m_2$. Recall each diagonal entry is equal to the number of edges incident to the corresponding black vertex.

For a moment we focus on one of the two boundaries, say $B_1$, and label its black vertices by $0,1,\ldots,m_1-1$ in order.

\begin{defn}[Neighbour counting]
\label{Def:NeighbourCounting}
For an integer $l$ with $0\leq l\leq m_1-1$ we denote by $e_l$ the number of neighbours of the black vertex $l$.
\end{defn}

\begin{defn}[Principal characteristic polynomials]
\label{Def:PrincChar}
For an integer $l$ with $1 \leq l\leq m_1$ we denote by $\charpoly{l}=\charpoly{l}(x)$ the characteristic polynomial of the principal submatrix of $B_1$ on rows and columns indexed by vertices $0,1,\ldots,l-1$. We use the convention $\charpoly{0}=1$.
\end{defn}

The following statement follows immediately from Lemma~\ref{Lemma:Tridiagonal}.

\begin{rem}
\label{Rem:CharPolyRecursion}
For any $l$ with $1 \leq l\leq m_1-1$ we have
\begin{align*}
\charpoly{l+1} =(x-e_l)\charpoly{l} - \charpoly{l-1}.
\end{align*}
\end{rem}

If $x_0$ is an eigenvalue of $B_1$, then $(\charpoly{l}(x_0))_{0\leq l\leq m_1-1}$ becomes an eigenvector of $B_1$ by Proposition~\ref{Prop:EntriesEigenvectors}. Remark~\ref{Rem:CharPolyRecursion} yields recursions for the entries of the eigenvector.

\subsection{Horizontal snake graphs}

Now we would like to apply the results from Section~\ref{Sec:Recursions} to various classes of snake graphs. As our first application, we reprove Kasteleyn and Temperley--Fisher's Theorem~\ref{Thm:Horizontal} about the eigenvalues of horizontal snake graphs $H_n$ before giving more original results in the subsequent subsections.

Let $n\geq 1$. Recall that the horizontal snake graph $H_n$ is formed from $n$ square tiles aligned in horizontal direction. It contains $2n+2$ vertices, $n+1$ on the upper boundary and $n+1$ on the lower boundary.

\begin{coro}
\label{Coro:Horizontal} Suppose that $G=H_n$. Then the characteristic polynomial of $BB^T$ equals
\label{Coro:PolyHorizontal}
\begin{align*}
    \charpoly{H_n}(x)=\begin{cases}
    (x-1)\fibQ_r(x)^2&\textrm{if $n=2r$ is even};\\
    \fibP_{r+1}(x)^2&\textrm{if $n=2r+1$ is odd}.\\
    \end{cases}
\end{align*}
\end{coro}

\begin{proof}
The graph $H_n$ has no turns, see Definition~\ref{Def:Turns}, because every vertex has exactly $3$ neighbours unless it is incident with the first or the last vertical edge. In particular, if we use any of the two ways to colour the graph in a bipartite fashion, every turn is black. So the matrix $BB^T$ is a $2\times 2$ block diagonal matrix with blocks indexed by the black vertices in the upper and lower boundary, respectively, according to Proposition~\ref{Prop:Tridiagonal}.

Suppose that $n=2r$ with $r\geq 1$. One of the diagonal blocks, say $B_1$, is associated to a boundary component having exactly $r$ black vertices each of which has exactly $3$ neighbours in $G$. So $\charpoly{B_1}(x)=\fibQ_r(x)$ by Definition~\ref{Def:FibProdPoly}. The other block, $B_2$, is associated to a boundary component with exactly $r+1$ black vertices. The first and the last black vertex have $2$ neighbours each, whereas the other vertices have $3$ neighbours each. Lemma~\ref{Lemma:Tridiagonal} implies $\charpoly{B_2}(x)=(x-2)\fibP_{r}(x)-\fibP_{r-1}(x)=\fibP_{r+1}(x)+\fibP_{r}(x)$, which is equal to $(x-1)\fibQ_r(x)$ by Lemma~\ref{Lemma:RelationPandQ}. Hence $\charpoly{H_n}(x)=\charpoly{B_1}(x)\charpoly{B_2}(x)=(x-1)\fibQ_r(x)^2$.

Now suppose that $n=2r+1$ with $r\geq 0$. Both diagonal blocks, $B_1$ and $B_2$, are associated to boundary components with exactly $r+1$ black vertices. On each boundary component there is one black vertex with $2$ neighbours, located at the beginning or at the end. The other vertices have $3$ neighbours each. It follows that $\charpoly{B_1}(x)=\charpoly{B_2}(x)=\fibP_{r+1}(x)$ which implies the claim.
\end{proof}

 Especially, the roots of the characteristic polynomial $\charpoly{H_n}$, except for the root $1$ in the case when $n$ is even, are equal to the roots of the Fibonacci product polynomials given in Proposition~\ref{Prop:RootsFibonacciProdPoly}. Their multiplicities are twice as large. Hence, Corollary~\ref{Coro:Horizontal} yields another proof of Kasteleyn and Temperley--Fisher's Theorem~\ref{Thm:Horizontal} and, together with Proposition~\ref{Prop:RootsFibonacciProdPoly}, it yields another proof of the Fibonacci product formula in Corollary~\ref{Coro:FibonacciProd}.

\subsection{ \lrev-shaped snake graphs}

In this subsection are interested in \lrev-shaped snake graphs. In every such graph there is exactly one turning tile having a $2$-turn and a $4$-turn. We colour the vertices of $G$ such that both turns are black. The matrix $BB^T$ is a block diagonal matrix, and specifically, the first and last diagonal entries in $B_1$ and $B_2$ will be either $2$ or $3$, and the entry at the index corresponding to the turn will be $2$ for the upper boundary and $4$ for the lower boundary.

\begin{theo}
\label{Theo:L}
The characteristic polynomials of $BB^T$ for the \lrev-shaped snake graphs are given by the expressions
\begin{align*}
\charpoly{\lrev_{2m,2n}}&=\left[\fibQ_{m+n-1}-\fibQ_{m-1}\fibQ_{n-1}\right]\left[(x-1)\fibQ_{m+n}+\fibP_m\fibP_n\right],\\
\charpoly{\lrev_{2m+1,2n+1}}&=\left[(x-1)\fibQ_{m+n}-\fibP_{m}\fibP_{n}\right]\left[\fibQ_{m+n+1}+\fibQ_{m}\fibQ_{n}\right],\\
\charpoly{\lrev_{2m,2n+1}}&=\left[\fibP_{m+n}-\fibQ_{m-1}\fibP_{n}\right]\left[\fibP_{m+n+1}+\fibP_{m}\fibQ_{n}\right].
\end{align*}
\end{theo}

\begin{proof}
\newcommand{\tmphtiles}{\text{width}}
\newcommand{\tmpvtiles}{\text{height}}
\newcommand{\tmpside}{S}
\newcommand{\tmpstart}{p}
\newcommand{\tmpend}{q}
We consider an \lrev-shaped graph $\lrev_{\tmphtiles,\tmpvtiles}$. Notice that $\tmpvtiles$ and $\tmphtiles$ can be equal to $1$ in which case $G$ becomes a horizontal or a vertical snake graph, respectively.

We consider either the upper or the lower boundary of the graph $\lrev_{\tmphtiles,\tmpvtiles}$, say the one corresponding to $B_1$ in the notation of Subsection~\ref{Subsec:BlackTurn}. This boundary component is a path graph which is naturally divided into a first (horizontal) leg and a second (vertical) leg. After reordering, we may assume that the black vertices on the horizontal leg are $0,1,\ldots,r-1$ and that the black vertices on the vertical leg are $r-1,\ldots,r+s-2$ for some $r,s\geq 1$. Here, the vertex $r-1$ is a turn unless $r=1$ or $s=1$ in which case the graph is horizontal or vertical and the vertex $r-1$ is the first or last black vertex in the boundary component.

 We define
\begin{multicols}{2}
\null \vfill
\noindent
\[\tmpside=\begin{cases}+1&\text{if on upper boundary};\\-1&\text{if on lower boundary};\end{cases}\]
\vfill \null
\noindent
\begin{align*}
\tmpstart&=\left(\tmphtiles+\frac{1-\tmpside}{2}\right)\operatorname{mod}2\in\{0,1\};\\
\tmpend&=\left(\tmpvtiles+\frac{1-\tmpside}{2}\right)\operatorname{mod}2\in\{0,1\}.
\end{align*}
\end{multicols}
\noindent We then also have
\begin{multicols}{2}
\noindent\[r=\ceil{\frac{\tmphtiles+\frac{1-\tmpside}{2}}{2}},\]
\noindent\[s=\ceil{\frac{\tmpvtiles+\frac{1-\tmpside}{2}}{2}}.\]
\end{multicols}

We follow the notation of Subsection~\ref{Subsec:BlackTurn} so that $e_0$ and $e_{r+s-2}$ are the number of white vertices in $G$ neighbouring the first and the last black vertex, respectively, and $e_{r-1}$ is the number of white vertices in $G$ neighbouring the turn (if it exists). It should be clear that
\begin{align*}
e_k=3
-\tmpstart\kdelta{k}{0}
+\tmpside\kdelta{k}{r-1}
-\tmpend\kdelta{k}{r+s-2}.
\end{align*}

Lemma~\ref{Lemma:TridiagFib} gives that the characteristic polynomial of $B_1$ is
\begin{align*}
\charpoly{r+s-1}&=\fibQ_{r+s-1}+\sum_{k=0}^{r+s-2}(3-e_k)\charpoly{k} \fibQ_{r+s-2-k}\\
&=\fibQ_{r+s-1}+\tmpstart\charpoly{0}\fibQ_{r+s-2}-\tmpside\charpoly{r-1}\fibQ_{s-1}+\tmpend\charpoly{r+s-2}\fibQ_0.
\end{align*}
Then we may repeat the use of Lemma~\ref{Lemma:TridiagFib} to get the following (provided we take $Q_{-1}=0$ so that the formula works for $r=1$ and/or $s=1$)
\begin{align*}
\charpoly{r-1}&=\fibQ_{r-1}+\tmpstart\charpoly{0}\fibQ_{r-2}=\fibQ_{r-1}+\tmpstart\fibQ_{r-2},\\
\charpoly{r+s-2}&=\fibQ_{r+s-2}+\tmpstart\charpoly{0}\fibQ_{r+s-3}-\tmpside\charpoly{r-1}\fibQ_{s-2}\\
&=\fibQ_{r+s-2}+\tmpstart\fibQ_{r+s-3}-\tmpside(\fibQ_{r-1}+\tmpstart\fibQ_{r-2})\fibQ_{s-2}.
\end{align*}
We then substitute these into the full characteristic polynomial to get
\begin{align*}
\charpoly{r+s-1}
&=\fibQ_{r+s-1}+(\tmpstart+\tmpend)\fibQ_{r+s-2}+\tmpstart\tmpend\fibQ_{r+s-3}
-\tmpside(\fibQ_{r-1}+p\fibQ_{r-2})(\fibQ_{s-1}+q\fibQ_{s-2}).
\end{align*}

The results from Lemma~\ref{Lemma:RelationPandQ} give for $k\geq0$ that $\fibP_k=\fibQ_k+\fibQ_{k-1}$ and
\begin{align*}
(x-1)\fibQ_k&=\fibP_{k+1}+\fibP_k=\fibQ_{k+1}+2\fibQ_k+\fibQ_{k-1},
\end{align*}
thus
\begin{align*}
    \fibQ_{r+s-1}+(\tmpstart+\tmpend)\fibQ_{r+s-2}+\tmpstart\tmpend\fibQ_{r+s-3}
  &=\begin{cases}
    \fibQ_{r+s-1}&\text{if $\tmpstart=\tmpend=0$};\\
    \fibP_{r+s-1}&\text{if $\{\tmpstart,\tmpend\}=\{0,1\}$};\\
    (x-1)\fibQ_{r+s-2}&\text{if $\tmpstart=\tmpend=1$};
    \end{cases}
\\
    \fibQ_{r-1}+\tmpstart\fibQ_{r-2}
  &=\begin{cases}
    \fibQ_{r-1}&\text{if $\tmpstart=0$};\\
    \fibP_{r-1}&\text{if $\tmpstart=1$};
    \end{cases}
\\
    \fibQ_{s-1}+\tmpend\fibQ_{s-2}
  &=\begin{cases}
    \fibQ_{s-1}&\text{if $\tmpend=0$};\\
    \fibP_{s-1}&\text{if $\tmpend=1$}.
    \end{cases}
\end{align*}
Therefore
\[
\charpoly{r+s-1} =
\begin{cases}
\fibQ_{r+s-1}-\tmpside\fibQ_{r-1}\fibQ_{s-1} &\text{if $\tmpstart=\tmpend=0$}; \\
\fibP_{r+s-1}-\tmpside\fibP_{r-1}\fibQ_{s-1} &\text{if $\tmpstart=1$, $\tmpend=0$}; \\
\fibP_{r+s-1}-\tmpside\fibQ_{r-1}\fibP_{s-1} &\text{if $\tmpstart=0$, $\tmpend=1$}; \\
(x-1)\fibQ_{r+s-2}-\tmpside\fibP_{r-1}\fibP_{s-1} &\text{if $\tmpstart=\tmpend=1$}.
\end{cases}
\]
This expression yields one of the two factors on the right hand side of each equation. The other factor is constructed similarly using the other boundary component of $G$.
\end{proof}

\begin{coro}
For every $n$ we have $\fibP_n\fibQ_{n-1} \mid \charpoly{\lrev_{2n,2n}}$ and $\fibP_n\fibQ_n \mid \charpoly{\lrev_{2n+1,2n+1}}$.
\end{coro}

\begin{proof}
 If we set $n=m$ in Theorem~\ref{Theo:L}, we get
 \begin{align*}
\charpoly{\lrev_{2n,2n}}&=\left[\fibQ_{2n-1}-\fibQ_{n-1}^2\right]\left[(x-1)\fibQ_{2n}+\fibP_n^2\right],\\
\charpoly{\lrev_{2n+1,2n+1}}&=\left[(x-1)\fibQ_{2n}-\fibP_{n}^2\right]\left[\fibQ_{2n+1}+\fibQ_{n}^2\right].
 \end{align*}
 Proposition~\ref{Prop:Divisible} establishes the claim.
\end{proof}

\subsection{Staircases} This subsection concerns the staircase $S_{m,3}$ with $m\geq 2$ as introduced in Example~\ref{Ex:SnakeGraphs}.

\begin{theo}
\label{Theo:Staircase}
\newcommand{\tmpvarB}{\left(\left(x-3\right)^2\right)}
The characteristic polynomials of $BB^T$ for the staircase snake graphs $S_{m,3}$ for $m\geq 2$ are given by the expressions
\begin{align*}
      \charpoly{S_{2k,3}}&=(x-2)\fibQ_k\tmpvarB\left[(x-2)\fibQ_k\tmpvarB+x\fibQ_{k-1}\tmpvarB\right],\\
    \charpoly{S_{2k+1,3}}&=\left[\fibQ_{k+1}\tmpvarB+(x-1)\fibQ_k\tmpvarB\right]^2.
\end{align*}
\end{theo}

\begin{proof}
\newcommand{\seqa}{a}
\newcommand{\seqb}{b}
\newcommand{\tmpvar}{u}
\newcommand{\tmpvarB}{\left(\left(x-3\right)^2\right)}
We define two infinite sequences of polynomials, $(\seqa_n)_{n\geq0}$ and $(\seqb_n)_{n\geq0}$, such that $\seqa_n$ is the characteristic polynomial of the $n\times n$ tridiagonal matrix as in Definition~\ref{Def:Tridiagonal} with alternating diagonal entries $(2,4,2,4,\ldots)$, and $\seqb_n$ is the same but with diagonal entries $(3,2,4,2,4\ldots)$ (where the $2,4$ part repeats).

We have $\seqa_0=1,\seqa_1=x-2,\seqb_0=1,\seqb_1=x-3$ and by Lemma~\ref{Lemma:Tridiagonal} we have for $k\geq 1$ the recursions
\begin{align*}
    \seqa_{2k}  &=(x-4)\seqa_{2k-1}-\seqa_{2k-2},&&&\seqb_{2k}  &=(x-2)\seqb_{2k-1}-\seqb_{2k-2},\\
    \seqa_{2k+1}&=(x-2)\seqa_{2k}-\seqa_{2k-1},&&&\seqb_{2k+1}&=(x-4)\seqb_{2k}-\seqb_{2k-1}.
\end{align*}
For $k\geq 1$ we can combine the two recursions for $\seqa_n$ as follows
\begin{align*}
    \seqa_{2k+2}&=(x-4)\seqa_{2k+1}-\seqa_{2k}\\
                &=(x-4)\left[(x-2)\seqa_{2k}-\seqa_{2k-1}\right]-\seqa_{2k}\\
                &=\left[(x-2)(x-4)-1\right]\seqa_{2k}-(x-4)\seqa_{2k-1}\\
                &=\left[(x-3)^2-1^2-1\right]\seqa_{2k}-(\seqa_{2k}+\seqa_{2k-2})\\
                &=\left((x-3)^2-3\right)\seqa_{2k}-\seqa_{2k-2}.
\end{align*}
Then with $\tmpvar(x)=(x-3)^2=x^2-6x+9$ we have $\seqa_{2k+2}=(u(x)-3)\seqa_{2k}-\seqa_{2k-2}$. Since we also have $\seqa_0=1=\fibP_0(\tmpvar(x))$ and $\seqa_2=x^2-6x+7=\tmpvar(x)-2=\fibP_1(\tmpvar(x))$, we can use the recursion in Remark~\ref{Rem:FibProdRec} to see that
\begin{align*}
    \seqa_{2k}=\fibP_k\left(\tmpvar(x)\right)=\fibP_k\tmpvarB.
\end{align*}
We then use the recursion, the above result, and the results from Lemma~\ref{Lemma:RelationPandQ} to get
\begin{align*}
    (x-4)\seqa_{2k+1}&=\seqa_{2k+2}+\seqa_{2k}\\
                     &=\fibP_{k+1}(\tmpvar(x))+\fibP_k(\tmpvar(x))\\
                     &=(\tmpvar(x)-1)\fibQ_k(\tmpvar(x))\\
                     &=(x-4)(x-2)\fibQ_k(\tmpvar(x))
\end{align*}
therefore we get
\begin{align*}
    \seqa_{2k+1}=(x-2)\fibQ_k(\tmpvar(x))=(x-2)\fibQ_k\tmpvarB.
\end{align*}
Next, notice that $\seqb_{2k+2}$ is also the characteristic polynomial of the $(2k+2)\times(2k+2)$ tridiagonal matrix as in Definition~\ref{Def:Tridiagonal} with diagonal entries (2,4,2,4,\ldots,2,3), hence for $k\geq 0$
\begin{align*}
    \seqb_{2k+2}&=(x-3)\seqa_{2k+1}-\seqa_{2k}\\
                &=(x-3)(x-2)\fibQ_k(\tmpvar(x))-\fibP_k(\tmpvar(x)).
\end{align*}
So for $k\geq 1$ we get
\begin{align*}
    \seqb_{2k}&=(x-3)(x-2)\fibQ_{k-1}(\tmpvar(x))-\fibP_{k-1}(\tmpvar(x))\\
                &=\left((x-3)^2+(x-3)\right)\fibQ_{k-1}(\tmpvar(x))-\fibQ_{k-1}(\tmpvar(x))-\fibQ_{k-2}(\tmpvar(x))\\
                &=\left(\tmpvar(x)-3\right)\fibQ_{k-1}(\tmpvar(x))-\fibQ_{k-2}(\tmpvar(x))+(x-1)\fibQ_{k-1}(\tmpvar(x))\\
                &=\fibQ_k(\tmpvar(x))+(x-1)\fibQ_{k-1}(\tmpvar(x)),
\end{align*}
Then, seperately verifying at $k=0$, we have for $k\geq 0$ that
\[ \seqb_{2k}=\fibQ_k(\tmpvar(x))+(x-1)\fibQ_{k-1}(\tmpvar(x)). \]
For the odd index $b_n$ terms we get for $k\geq 1$
\begin{align*}
    (x-2)\seqb_{2k+1}={}& \seqb_{2k+2}+\seqb_{2k}\\
                     ={}& (x-3)(x-2)\fibQ_k(\tmpvar(x))-\fibP_k(\tmpvar(x))\\ &+(x-3)(x-2)\fibQ_{k-1}(\tmpvar(x))-\fibP_{k-1}(\tmpvar(x))\\
                     ={}& (x-3)(x-2)(\fibQ_k(\tmpvar(x))+\fibQ_{k-1}(\tmpvar(x)))-(\tmpvar(x)-1)\fibQ_{k-1}(\tmpvar(x))\\
                     ={}& (x-3)(x-2)(\fibQ_k(\tmpvar(x))+\fibQ_{k-1}(\tmpvar(x)))-(x-4)(x-2)\fibQ_{k-1}(\tmpvar(x))\\
                     ={}& (x-2)[(x-3)\fibQ_k(\tmpvar(x))+\fibQ_{k-1}(\tmpvar(x))].
\end{align*}
Therefore, again separately verifying at $k=0$, we have for $k\geq 0$ that
\[ \seqb_{2k+1}=(x-3)\fibQ_k(\tmpvar(x))+\fibQ_{k-1}(\tmpvar(x)). \]

We now consider a staircase graph of the form $S_{m,3}$. In the notation of Subsection~\ref{Subsec:BlackTurn} we suppose that $B_1$ is the block of $BB^T$ parametrised by the upper boundary and denote by $\charpoly{l}^+=\charpoly{l}^+(x)$ the characteristic polynomials of its submatrices as in Definition~\ref{Def:PrincChar}. We denote by $\charpoly{l}^-$ the corresponding characteristic polynomials of submatrices of $B_2$. There are $m+1$ black vertices on each of the upper and lower boundaries, so $B_1$ and $B_2$ are both $(m+1)\times(m+1)$ matrices. In fact $B_1$ and $B_2$ are both tridiagonal matrices of the form given in Definition~\ref{Def:Tridiagonal}.

All diagonal entries of $B_1$ except for the last follow the alternating pattern $(2,4,2,4,\ldots)$, and for $B_2$ the pattern is $(3,2,4,2,4,\ldots)$ (where the $2,4$ part repeats), so for $0\leq n\leq m$ we have the following
\begin{alignat*}{2}
    \charpoly{n}^+&{}={}&&\seqa_n=\begin{cases}
    \fibP_k\tmpvarB     &\text{if $n=2k$ even};\\
    (x-2)\fibQ_k\tmpvarB&\text{if $n=2k+1$ odd};
    \end{cases}\\
    \charpoly{n}^-&{}={}&&\seqb_n=\begin{cases}
    \fibQ_k\tmpvarB+(x-1)\fibQ_{k-1}\tmpvarB&\text{if $n=2k$ even};\\
    (x-3)\fibQ_k\tmpvarB+\fibQ_{k-1}\tmpvarB&\text{if $n=2k+1$ odd}.
    \end{cases}
\end{alignat*}
The last diagonal entry of $B_1$ is $2$ if $m$ is even (hence the staircase $S_{m,3}$ ends with a vertical segment) and $3$ if $m$ is odd ($S_{m,3}$ ends with a horizontal segment), and similarly the last diagonal entry of $B_2$ is $3$ if $m$ is even and $2$ if $m$ is odd. We use this and the recursion from Remark~\ref{Rem:CharPolyRecursion} (or Lemma~\ref{Lemma:Tridiagonal}) to get
\begin{align*}
    \charpoly{m+1}^+&=\begin{cases}
    (x-2)\seqa_{2k}-\seqa_{2k-1}=\seqa_{2k+1}&\text{if $m=2k$ even};\\
    (x-3)\seqa_{2k+1}-\seqa_{2k}=\seqb_{2k+2}&\text{if $m=2k+1$ odd};
    \end{cases}\\
    \charpoly{m+1}^-&=\begin{cases}
    (x-3)\seqb_{2k}-\seqb_{2k-1}=\seqb_{2k+1}+\seqb_{2k}&\text{if $m=2k$ even};\\
    (x-2)\seqb_{2k+1}-\seqb_{2k}=\seqb_{2k+2}&\text{if $m=2k+1$ odd};
    \end{cases}
\end{align*}
thus, using the expressions we found for the $\seqa_n$ and $\seqb_n$ terms, we get
\begin{align*}
    \charpoly{m+1}^+&=\begin{cases}
    (x-2)\fibQ_k\tmpvarB&\text{if $m=2k$ even};\\
    \fibQ_{k+1}\tmpvarB+(x-1)\fibQ_k\tmpvarB&\text{if $m=2k+1$ odd};
    \end{cases}\\
    \charpoly{m+1}^-&=\begin{cases}
    (x-2)\fibQ_k\tmpvarB+x\fibQ_{k-1}\tmpvarB&\text{if $m=2k$ even};\\
    \fibQ_{k+1}\tmpvarB+(x-1)\fibQ_k\tmpvarB&\text{if $m=2k+1$ odd}.
    \end{cases}
\end{align*}
Finally, the full characteristic polynomial is given by the product $\charpoly{m+1}^+\charpoly{m+1}^-$.
\end{proof}

\section{Characteristic polynomials via continued fractions}

\subsection{Convergents of continued fractions}

In this section we use a continued fraction expansion as in Equation~(\ref{Eqn:NegConFrac}) to construct the characteristic polynomial of $BB^T$.

Suppose that $(a_1,\ldots,a_{k})$ is a sequence of positive integers or a sequence of non-constant polynomials in the ring $\mathbb{Q}[x]$.

\begin{defn}[Convergents]
\label{Def:Convergents}
We define two sequences $(p_0,\ldots,p_k)$ and $(q_0,\ldots,q_k)$ by the recursions $p_0=1$, $p_1=a_1$ and $p_{l}=a_lp_{l-1}+p_{l-2}$ for $l\in [2,k]$ and $q_0=0$, $q_1=1$ and $q_{l}=a_lq_{l-1}+q_{l-2}$ for $l\in [2,k]$. The quotients $p_l/q_l$ are called the \emph{convergents} of the continued fraction $[a_1,\ldots,a_k]$.
\end{defn}

Note that the elements $p_l$ and $q_l$ in Definition~\ref{Def:Convergents} are non-zero thanks to our assumptions on the sequence $(a_l)_{1\leq l\leq k}$. The following proposition is well-known.

\begin{prop} Let $l\in [1,k]$. The convergents satisfy
\begin{align*}
[a_1,a_2,\ldots,a_l]=a_1+\cfrac{1}{a_2 +\cfrac{1}{a_3 +\cfrac{1}{\ddots + \cfrac{1}{a_l}}}}=\frac{p_l}{q_l}.
\end{align*}
\end{prop}

Here, the quotient $p_l/q_l$ is a positive rational number or an element in the field of fractions $\mathbb{Q}(x)$. Moreover, using the recursion it is easy to show that $p_lq_{l-1}-q_lp_{l-1}=\pm 1$ for every $l$. This equation implies that the numerator $p_l$ and the denominator $q_l$ are coprime.

\subsection{Characteristic polynomials as numerators of continued fractions}
\label{Subsec:ContinuedFracWhiteTurns}

Suppose that $G$ is a snake graph coloured such that all turns are black. We equip the graph with the Kasteleyn weighting from Proposition~\ref{Prop:Tridiagonal}. So the matrix $BB^T$ is a $2\times 2$ block diagonal matrix whose diagonal blocks $B_1$ and $B_2$ are indexed by the upper and lower boundary. We pick one of the boundaries, say the one corresponding to $B_1$, and label its black vertices by $1,2,\ldots,k$ in order. (Notice that in contrast to previous sections we begin the indexing with $1$ for technical reasons.)

\begin{defn}[Linear constituents]
For an integer $l\in [1,k]$ we put $c_l(x)=x-e_l\in \mathbb{Z}[x]$ where, as before, $e_l$ is the number of neighbours of the black vertex $l$ in $G$.
\end{defn}

\begin{theo}
\label{Theo:Numerator}
Let $G$ be a snake graph whose turns are all black. For any Kasteleyn weighting the characteristic polynomial of the weighted adjacency matrix $A$ can be written as $\charpoly{A}(t)=\charpoly{B_1}(t^2)\charpoly{B_2}(t^2)$
where $\charpoly{B_1}(x)=p(x)$ is the numerator of the continued fraction
\begin{align*}
    [[c_1(x),\ldots,c_k(x)]]=c_1(x)-\cfrac{1}{c_2(x) -\cfrac{1}{c_3(x) -\cfrac{1}{\ddots - \cfrac{1}{c_k(x)}}}}=\frac{p(x)}{q(x)}
\end{align*}
and the polynomial $\charpoly{B_2}(x)$ is constructed similarly using the second boundary graph.
\end{theo}

\begin{proof}
Let $l\in [1,k]$. We consider the characteristic polynomial $\charpoly{l}=\charpoly{l}(x)$ of the principal submatrix of $B_1$ on rows $[1,l]$, see Defintion~\ref{Def:PrincChar}. We claim that $\charpoly{l}(x)=\varepsilon(l)p_l(x)$ where $p_l$ is the numerator of the continued fraction
\begin{align*}
[[c_1(x),c_2(x),\ldots,c_l(x) ]]=[c_1(x),-c_2(x),c_3(x),\ldots,(-1)^{l+1}c_l(x)]
\end{align*}
and $\varepsilon(l)\in\{\pm 1\}$ is a sign. More precisely, we claim that $\varepsilon\colon\mathbb{N}\to \{\pm 1\}$ is $4$-periodic with initial values $(\varepsilon(0),\varepsilon(1),\varepsilon(2),\varepsilon(3))=(1,1,-1,-1)$; in other words, $\varepsilon(l)=(-1)^{\sigma(l)}$ with $\sigma(l)=\tfrac{1}{4}[2l+(-1)^l-1]$ for all $l\geq 0$. Note that $\varepsilon(l)=-\varepsilon(l-2)$ and $\varepsilon(l)=(-1)^{l+1}\varepsilon(l-1)$ for all $l\geq 2$.

We prove the claim by induction. The claim is true for $l=1$ since $\charpoly{1}(x)=x-e_1$ and for $l=2$ since $\charpoly{2}(x)=x^2-(e_1+e_2)x+e_1e_2-1$. Remark~\ref{Rem:CharPolyRecursion} asserts that $\charpoly{l}(x)=c_l(x)\charpoly{l-1}(x)-\charpoly{l-2}(x)$ for $l \in [2,k]$. The right hand side of the last expression becomes $\varepsilon(l)[(-1)^{l+1}c_l(x)p_{l-1}(x)+p_{l-2}]$ if we apply the induction hypothesis. This is equal to $\varepsilon(l)p_l(x)$ by Definition~\ref{Def:Convergents}.

An application of Proposition~\ref{Prop:CharPoly} finishes the proof of the statement.
\end{proof}

\begin{exam} The lower boundary of the graph in Figure~\ref{Fig:NeighbourCount} has a continued fraction
    \begin{align*}
x-3-\cfrac{1}{x-2 -\cfrac{1}{x-4-\cfrac{1}{x-3}}}=\frac{x^4-12x^3+50x^2-84x+46}{x^3-9x^2+24x-19};
\end{align*}
the numerator generates
\begin{align*}
    \charpoly{B_1}(t)=\left(x^4-12x^3+50x^2-84x+46\right)\biggr\rvert_{x=t^2}=t^8-12t^6+50t^4-84t^2+46.
\end{align*}
\end{exam}

\begin{figure}

\begin{center}
\begin{tikzpicture}[every path/.style={dotted}]

\pgfdeclarelayer{background}
\pgfdeclarelayer{foreground}
\pgfsetlayers{background,main,foreground}

\definecolor{c1}{RGB}{150,150,150}

\newcommand{\x}{1.5cm} %%% side length of a tile
\newcommand{\R}{0.5pt}  %%% size of bullet
\newcommand{\s}{0.5cm}  %%% shift of labels
\newcommand{\sos}{0.2cm}  %%% shift of labels (signs)

\draw (0,0) node[fill,circle,minimum size=\R,color=c1]{};
\node at (0.5*\x,0*\x+\sos) {$\bm{+}$};
\draw (0,\x) node[fill,circle,minimum size=\R]{};

\draw (\x,0) node[fill,circle,minimum size=\R]{};
\node at (1*\x,0*\x-\s) {$3$};
\draw (\x,\x) node[fill,circle,minimum size=\R,color=c1]{};
\node at (\x+\sos,0.5*\x) {$\bm{+}$};

\draw (2*\x,0) node[fill,circle,minimum size=\R,color=c1]{};
\node at (2*\x+\sos,0.5*\x) {$\bm{-}$};
\draw (2*\x,\x) node[fill,circle,minimum size=\R]{};

\draw (3*\x,0) node[fill,circle,minimum size=\R]{};
\node at (3*\x,0*\x-\s) {$2$};
\draw (3*\x,\x) node[fill,circle,minimum size=\R,color=c1]{};
\node at (2.5*\x,\x+\sos) {$\bm{-}$};

\draw (2*\x,2*\x) node[fill,circle,minimum size=\R,color=c1]{};
\draw (3*\x,2*\x) node[fill,circle,minimum size=\R]{};
\node at (3*\x+\s,2*\x-\s) {$4$};
\node at (2.5*\x,2*\x+\sos) {{$\bm{+}$}};

\draw (2*\x,3*\x) node[fill,circle,minimum size=\R]{};
\draw (3*\x,3*\x) node[fill,circle,minimum size=\R,color=c1]{};
\node at (3*\x+\sos,2.5*\x) {{$\bm{+}$}};

\draw (4*\x,2*\x) node[fill,circle,minimum size=\R,color=c1]{};
\draw (4*\x,3*\x) node[fill,circle,minimum size=\R]{};
\node at (4*\x+\sos,2.5*\x) {{$\bm{-}$}};

\draw (5*\x,2*\x) node[fill,circle,minimum size=\R]{};
\node at (5*\x,2*\x-\s) {$3$};
\draw (5*\x,3*\x) node[fill,circle,minimum size=\R,color=c1]{};
\node at (5*\x+\sos,2.5*\x) {{$\bm{+}$}};

\draw (6*\x,2*\x) node[fill,circle,minimum size=\R,color=c1]{};
\draw (6*\x,3*\x) node[fill,circle,minimum size=\R]{};
\node at (6*\x+\sos,2.5*\x) {{$\bm{-}$}};

\begin{pgfonlayer}{background}
\path[draw] (0,0) to (3*\x,0);
\path[draw] (0,\x) to (3*\x,\x);

\path[draw] (2*\x,0) to (2*\x,3*\x);
\path[draw] (3*\x,0) to (3*\x,3*\x);

\path[draw] (2*\x,2*\x) to (6*\x,2*\x);
\path[draw] (2*\x,3*\x) to (6*\x,3*\x);

\path[draw] (4*\x,2*\x) to (4*\x,3*\x);
\path[draw] (5*\x,2*\x) to (5*\x,3*\x);
\path[draw] (6*\x,2*\x) to (6*\x,3*\x);

\path[draw] (0,0) to (0,\x);
\path[draw] (\x,0) to (\x,\x);

\end{pgfonlayer}

\end{tikzpicture}
\end{center}
     \caption{The neighbour count on the lower boundary}
     \label{Fig:NeighbourCount}
\end{figure}
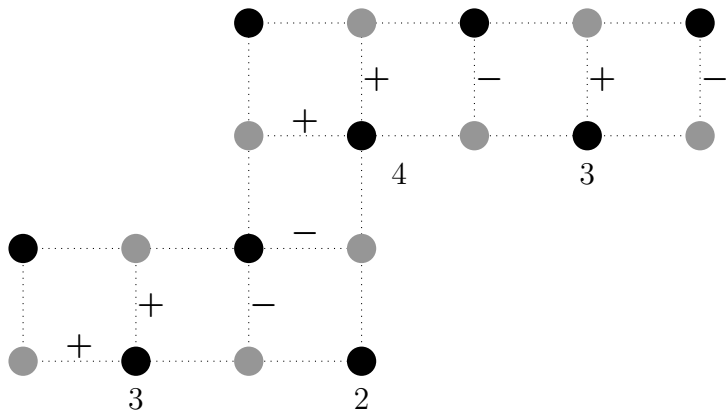

\subsection{Generalised snake graphs via rotations}
In Subsection~\ref{Subsec:ContinuedFracWhiteTurns} we have given a method to construct the characteristic polynomial when all turns of the given snake graph are black. In this subsection we present a way to transform a snake graph with white turns into a generalised snake graph and explain how to apply the results from Subsection~\ref{Subsec:ContinuedFracWhiteTurns} to generalised snake graphs.

\begin{defn}[Rotation at squares]
\label{Def:Rotation}
Let $G=(V,E)$ be a bipartite graph whose vertices can be partitioned in $V=V_0\sqcup V_1\sqcup V_2$ such that the following conditions hold.
\begin{itemize}
    \item[(i)] The full subgraph $T$ of $G$ on vertices $V_0$ is a square with vertices $A,B,C,D$.
    \item[(ii)] There are vertices $S_1,\ldots,S_n\in V_2$ and $R_1,\ldots,R_m\in V_2$ such that the edges $(D,S_i)$ for $i\in [1,n]$ and $(C,R_j)$ for $j\in [1,m]$ are the only edges in $G$ connecting a vertex in $V_2$ to a vertex in $V\backslash V_2$.
    \item[(iii)] A condition analogous to \textnormal{(ii)} holds for the full subgraph $G_1$ of $G$ on vertices $V_1$, see Figure~\ref{Fig:Rotation}.
\end{itemize}
The \emph{rotation} of $G$ at $T$ is the graph $r_T(G)=G'=(V,E')$ with the same vertices and
\begin{align*}
    E'=\left(E\backslash \left\lbrace (D,S_i),(C,R_j) \colon i\in [1,n], j\in [1,m]\right\rbrace\right) \cup \left\lbrace (B,S_i),(A,R_j)\colon i\in[1,n], j\in [1,m]\right\rbrace.
\end{align*}
\end{defn}

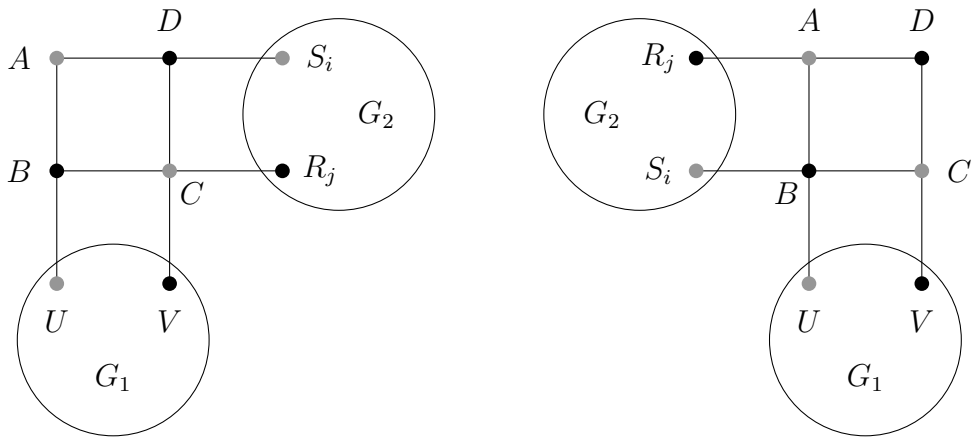
\begin{figure}
\begin{center}

\begin{tikzpicture}%[every path/.style={dotted}]

\pgfdeclarelayer{background}
\pgfdeclarelayer{foreground}
\pgfsetlayers{background,main,foreground}

\definecolor{c1}{RGB}{150,150,150}

\newcommand{\x}{1.5cm} %%% side length of a tile
\newcommand{\R}{0.5}  %%% size of bullet
\newcommand{\s}{0.5cm}  %%% shift of labels
\newcommand{\sos}{0.5cm}  %%% shift of labels (signs)
\newcommand{\rs}{5}  %%% number of rays
\newcommand{\sh}{10cm}  %%% distance between figures

\draw (0,0) node[fill,circle,scale=\R,color=c1]{};
\draw (0.6*\sos,-0.6*\sos) node{$C$};
\draw (0,\x) node[fill,circle,scale=\R]{};
\draw (0,\x+\sos) node{$D$};
\draw (-\x,0) node[fill,circle,scale=\R]{};
\draw (-\sos-\x,0) node{$B$};
\draw (-\x,\x) node[fill,circle,scale=\R,color=c1]{};
\draw (-\sos-\x,\x) node{$A$};

\draw (0,-\x-\sos) node{$V$};
\draw (-\x,-\x-\sos) node{$U$};

\draw (\x,0) node[fill,circle,scale=\R]{};
\draw (\x+\sos,0) node{$R_j$};
\draw (\x,\x) node[fill,circle,scale=\R,color=c1]{};
\draw (\x+\sos,\x) node{$S_i$};

\draw (0,-\x) node[fill,circle,scale=\R]{};
\draw (-\x,-\x) node[fill,circle,scale=\R,color=c1]{};

\draw (1.5*\x+\sos,0.5*\x) node{$G_2$};
\draw (-0.5*\x,-1.5*\x-\sos) node{$G_1$};

\draw (\sh,0) node[fill,circle,scale=\R,color=c1]{};
\draw (\sh+\sos,0) node{$C$};
\draw (\sh,\x) node[fill,circle,scale=\R]{};
\draw (\sh,\x+\sos) node{$D$};
\draw (\sh-\x,0) node[fill,circle,scale=\R]{};
\draw (\sh-\x-0.6*\sos,-0.6*\sos) node{$B$};
\draw (\sh-\x,\x) node[fill,circle,scale=\R,color=c1]{};
\draw (\sh-\x,\x+\sos) node {$A$};

\draw (\sh-2*\x,\x) node[fill,circle,scale=\R]{};
\draw (\sh-2*\x-\sos,0) node{$S_i$};
\draw (\sh-2*\x,0) node[fill,circle,scale=\R,color=c1]{};
\draw (\sh-2*\x-\sos,\x) node{$R_j$};

\draw (\sh,-\x) node[fill,circle,scale=\R]{};
\draw (\sh-\x,-\x) node[fill,circle,scale=\R,color=c1]{};

\draw (\sh-2.5*\x-\sos,0.5*\x) node{$G_2$};
\draw (\sh-0.5*\x,-1.5*\x-\sos) node{$G_1$};

\begin{pgfonlayer}{background}
\path[draw] (0,0) to (-\x,0);
\path[draw] (-\x,0) to (-\x,\x);
\path[draw] (-\x,\x) to (0,\x);
\path[draw] (0,\x) to (0,0);

%\path[draw] (\x,0) to (\x,\x);
\path[draw] (\x,\x) to (0,\x);
\path[draw] (\x,0) to (0,0);

%\path[draw] (0,-\x) to (-\x,-\x);
\path[draw] (-\x,-\x) to (-\x,0);
\path[draw] (0,-\x) to (0,0);

\draw (1.5*\x,0.5*\x) circle (0.85*\x);
\draw (-0.5*\x,-1.5*\x) circle (0.85*\x);

\path[draw] (\sh,0) to (\sh-\x,0);
\path[draw] (\sh-\x,0) to (\sh-\x,\x);
\path[draw] (\sh-\x,\x) to (\sh,\x);
\path[draw] (\sh,\x) to (\sh,0);

\path[draw] (\sh-\x,0) to (\sh-2*\x,0);
\path[draw] (\sh-2*\x,\x) to (\sh-\x,\x);
%\path[draw] (\sh-2*\x,\x) to (\sh-2*\x,0);

%\path[draw] (\sh,-\x) to (\sh-\x,-\x);
\path[draw] (\sh-\x,-\x) to (\sh-\x,0);
\path[draw] (\sh,-\x) to (\sh,0);

\draw (\sh,-\x-\sos) node{$V$};
\draw (\sh-\x,-\x-\sos) node{$U$};

\draw (\sh-2.5*\x,0.5*\x) circle (0.85*\x);
\draw (\sh-0.5*\x,-1.5*\x) circle (0.85*\x);

\end{pgfonlayer}

\end{tikzpicture}
\end{center}
     \caption{A rotation}
     \label{Fig:Rotation}
\end{figure}

\begin{lemma}
\label{Lemma:Rotate}
In the setup of Definition~\ref{Def:Rotation} the graphs $G=(V,E)$ and $r_T(G)=(V,E')$ have the same number of perfect matchings.
\end{lemma}

\begin{proof}
Let $\operatorname{Match}(G)$ denote the set of perfect matchings of $G$.
We define a map 
\begin{align*}
\Phi\colon \operatorname{Match}(G)\to \operatorname{Match}(G')
\end{align*}
as follows. Consider a perfect matching $m\in \PM(G)$. Note that the vertex $A$ can be matched in two ways.
\begin{itemize}
    \item[(i)] The first way is $(A,D)\in m$. In this case $(D,S_i)\notin m$ for all $i\in [1,n]$. Moreover, $(C,R_j)\notin m$ for any $j\in [1,m]$ for parity reasons. Hence, $m\subseteq E'$ is a perfect matching of $G'$ and we put $\Phi(m)=m$.
    \item[(ii)] The second way is $(A,B)\in m$.
    \begin{itemize}
        \item[(iia)] If $(C,D)\in m$, then $m\subseteq E'$ and we put $\Phi(m)=m$.
        \item[(iib)] If $(C,D)\notin m$, then $(D,S_i)\in m$ for some $i\in [1,n]$ because $D$ must be matched. Furthermore, $(C,R_j)\in m$ for some $j\in [1,m]$ for parity reasons. We define
        \begin{align*}
            \Phi(m)=\left(m \backslash \left\lbrace\right (A,B),(D,S_i),(C,R_j)\rbrace\right) \cup \left\lbrace (C,D),(A,R_j),(B,S_i)\right\rbrace.
        \end{align*}
    \end{itemize}
\end{itemize}
It is easy to see that $\Phi$ is bijective.
\end{proof}

\begin{defn}[Generalised snake graphs]
A \emph{generalised snake graph} is a graph that can be obtained from a snake graph by applying a sequence of rotations.
\end{defn}

\begin{rem}
Every generalised snake graph is a concatenation of a sequence of square tiles such that each tile is glued to the previous tile along an edge. Moreover, every generalised snake graph is planar.
\end{rem}

Let us generalise the notions from Subsection~\ref{Subsec:SpecialVertices}.

\begin{defn}[Turns]
 Let $G$ be a generalised snake graph. A vertex of $G$ is called a \emph{turn} if one of the following conditions holds.
\begin{itemize}
    \item[(i)] It is a vertex of at least $3$ tiles.
    \item[(ii)] It is diagonally opposite (across a tile) a turn as defined in (i), is a vertex of exactly one tile, and is neither a vertex of the first tile nor a vertex of the last tile.
\end{itemize}
A turn is called an \emph{$n$-turn} if it is adjacent to exactly $n$ vertices (so that $n\in\{2,4\}$ if G is a snake graph).
A tile of G with exactly two turns as vertices is called a \emph{turning tile}.
\end{defn}
Note that $2$ is the maximum number of turns a tile can have, that a turn defined by condition (ii) is always a $2$-turn, and that a turn defined by condition (i) is never a $2$-turn.
No turn is ever a $3$-turn, and for $n>3$ any $n$-turn is diagonally opposite exactly $n-3$ $2$-turns.

The proof of the following proposition is straightforward.

\begin{prop}
\label{Prop:Blackening}
Given a snake graph $G$, then there exists a sequence of rotations that transforms $G$ into a generalised snake graph whose turns are all black.
\end{prop}

\newcommand{\black}{\operatorname{black}}

\begin{exam}
Figure~\ref{Fig:Equivalent} shows the staircase $S_{6,2}$ and a generalised snake graph whose turns are all black that can be obtained from $S_{6,2}$ by a sequence of rotations.
\end{exam}

The generalised snake graph from Proposition~\ref{Prop:Blackening} is unique and we denote it by $\black(G)$.

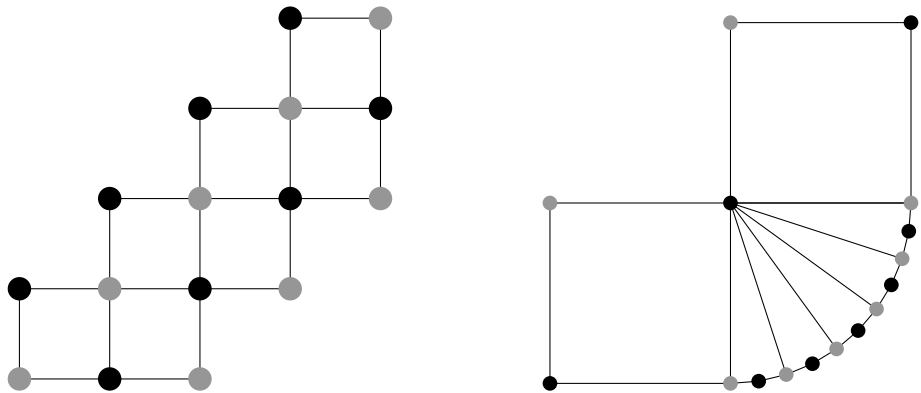
\begin{figure}

\begin{center}
\begin{tikzpicture}%[every path/.style={dotted}]

\pgfdeclarelayer{background}
\pgfdeclarelayer{foreground}
\pgfsetlayers{background,main,foreground}

\definecolor{c1}{RGB}{150,150,150}

\newcommand{\x}{1.2cm} %%% side length of a tile
\newcommand{\R}{0.8}  %%% size of bullet
\newcommand{\s}{0.5cm}  %%% shift of labels
\newcommand{\sos}{0.5cm}  %%% shift of labels (signs)

\draw (0,0) node[fill,circle,scale=\R,color=c1]{};

%\draw (-\sos,\x) node{$A$};
%\draw (-\sos,0) node{$B$};

\draw (0,\x) node[fill,circle,scale=\R]{};

\draw (\x,\x) node[fill,circle,color=c1,scale=\R]{};
\draw (\x,0) node[fill,circle,scale=\R]{};

\draw (2*\x,0) node[fill,circle,scale=\R,color=c1]{};
\draw (2*\x,\x) node[fill,circle,scale=\R]{};

\draw (2*\x,2*\x) node[fill,circle,scale=\R,color=c1]{};
\draw (\x,2*\x) node[fill,circle,scale=\R]{};

\draw (3*\x,\x) node[fill,circle,scale=\R,color=c1]{};
\draw (3*\x,2*\x) node[fill,circle,scale=\R]{};

\draw (3*\x,3*\x) node[fill,circle,scale=\R,color=c1]{};
\draw (2*\x,3*\x) node[fill,circle,scale=\R]{};

\draw (4*\x,2*\x) node[fill,circle,scale=\R,color=c1]{};
\draw (4*\x,3*\x) node[fill,circle,scale=\R]{};

\draw (4*\x,4*\x) node[fill,circle,scale=\R,color=c1]{};
\draw (3*\x,4*\x) node[fill,circle,scale=\R]{};

%\draw (4*\x,4*\x+\sos) node{$D$};
%\draw (3*\x,4*\x+\sos) node{$C$};

%\draw (5*\x,3*\x) node[fill,circle,scale=\R,color=c1]{};
%\draw (5*\x,4*\x) node[fill,circle,scale=\R]{};

\begin{pgfonlayer}{background}
\path[draw] (0,0) to (2*\x,0);
\path[draw] (0,\x) to (3*\x,\x);
\path[draw] (0,0) to (0,\x);

\path[draw] (\x,0) to (\x,2*\x);
\path[draw] (2*\x,0) to (2*\x,3*\x);

\path[draw] (\x,2*\x) to (4*\x,2*\x);
\path[draw] (3*\x,\x) to (3*\x,4*\x);

\path[draw] (2*\x,3*\x) to (4*\x,3*\x);
\path[draw] (3*\x,4*\x) to (4*\x,4*\x);
\path[draw] (4*\x,2*\x) to (4*\x,4*\x);

%\path[draw] (5*\x,3*\x) to (5*\x,4*\x);
\end{pgfonlayer}

\end{tikzpicture}\hspace{2cm}\begin{tikzpicture}%[every path/.style={dotted}]

\pgfdeclarelayer{background}
\pgfdeclarelayer{foreground}
\pgfsetlayers{background,main,foreground}

\definecolor{c1}{RGB}{150,150,150}

\newcommand{\x}{2.4cm} %%% side length of a tile
\newcommand{\R}{0.5}  %%% size of bullet
\newcommand{\s}{0.5cm}  %%% shift of labels
\newcommand{\sos}{0.5cm}  %%% shift of labels (signs)
\newcommand{\rs}{5}  %%% number of rays

\draw (0,0) node[fill,circle,scale=\R]{};
\draw (-\x,0) node[fill,circle,scale=\R,color=c1]{};
%\draw (-\x-\sos,0) node{$A$};

\draw (0,-\x) node[fill,circle,scale=\R,color=c1]{};
\draw (-\x,-\x) node[fill,circle,scale=\R]{};
%\draw (-\x-\sos,-\x) node{$B$};

\draw (\x,0) node[fill,circle,scale=\R,color=c1]{};
\draw (0,\x) node[fill,circle,scale=\R,color=c1]{};
\draw (\x,\x) node[fill,circle,scale=\R]{};
%\draw (0,\x+\sos) node{$C$};
%\draw (\x,\x+\sos) node{$D$};

\foreach \number in {2,...,\rs}{
\draw (270-90/\rs+90*\number/\rs:\x) node[fill,circle,scale=\R,color=c1]{};
}

\foreach \number in {1,...,\rs}{
\draw (270-45/\rs+90*\number/\rs:\x) node[fill,circle,scale=\R]{};
}

\begin{pgfonlayer}{background}
\path[draw] (0,0) to (-\x,0);
\path[draw] (-\x,0) to (-\x,-\x);
\path[draw] (-\x,-\x) to (0,-\x);
\path[draw] (0,-\x) to (0,0);

\path[draw] (0,0) to (\x,0);
\path[draw] (\x,0) to (\x,\x);
\path[draw] (\x,\x) to (0,\x);
\path[draw] (0,\x) to (0,0);

\foreach \number in {1,...,\rs}{
\path[draw] (270+90*\number/\rs:\x) to (0,0);
\path[draw] (270-45/\rs+90*\number/\rs:\x) to (270-90/\rs+90*\number/\rs:\x);
\path[draw] (270-45/\rs+90*\number/\rs:\x) to (270+90*\number/\rs:\x);
}

\end{pgfonlayer}

\end{tikzpicture}
\end{center}
     \caption{Transformation into a graph with black turns}
     \label{Fig:Equivalent}
\end{figure}

\subsection{Continued fractions for graphs with white turns}
To generalise Theorem~\ref{Theo:Numerator} we refine the notion of upper and lower boundary. Suppose that $G$ is a generalised snake graph.

\newcommand{\BE}{\operatorname{B}}

Suppose that $G$ has at least two tiles. As in Subsection~\ref{Subsec:SpecialVertices}, we will refer to the edges that bound the infinite face as boundary edges. The \emph{start edge} of $G$ is the edge of the first square tile connecting the $2$ vertices that do not belong to any other tile. The \emph{end edge} is defined similarly for the last tile. The subgraph of $G$ formed by the boundary edges is a circular graph $C(G)$.

If we further remove the start and the end edge from the circular graph $C(G)$, then we obtain a disjoint union of two path graphs which we will call the \emph{upper} and \emph{lower boundary} of $G$ as in Definition~\ref{Def:UpperLowerBoundary}. Here, the path graph containing the lower left vertex of the first tile is the upper boundary if the next tile is above the first tile, otherwise it is the lower boundary. Vertices in the upper or lower boundary are called \emph{upper} and \emph{lower} vertices, respectively.

Musiker, Schiffler and Williams~\cite{MS10,MSW11} construct snake graphs from triangulated surfaces. In this article we consider triangulated polygons since every snake graph can be realised as the snake graph of a diagonal inside a triangulated polygon. Here, a triangulation of a convex polygon is a maximal set of diagonals that do not intersect except possibly at endpoints.

We fix a convex polygon $\widetilde{\Sigma}$ and a triangulation $\widetilde{T}$ of $\widetilde{\Sigma}$. Let $\gamma$ be a diagonal in $\widetilde{\Sigma}$ that is not part of the triangulation. Let $\Sigma\subseteq \widetilde{\Sigma}$ be the union of the triangles of $\widetilde{T}$ that intersect $\gamma$. Then $\Sigma$ is again a polygon, and it is naturally triangulated by a set $T\subseteq \widetilde{T}$. We label the vertices of $\Sigma$ in clockwise order by $0,1,\ldots,m-1$ for some $m\geq 3$, and we will let $k\geq 1$ denote the natural number such that $\gamma$ connects vertex $0$ to vertex $k+1$. Let $G=G_{\gamma}$ be the snake graph associated with $\gamma$.

Let $(a_1,\ldots,a_n)$ be the sign sequence of $G$. Without loss of generality we may assume that $n$ is even after applying the procedure in Remark~\ref{Rem:Unique} if necessary. (This procedure corresponds to reversing the choice of the internal edge in the last tile.) We use Morier-Genoud and Ovsienko's formula \cite[Equation (1.2)]{MO19} to define a sequence
\begin{align*}
(e_1,\ldots,e_{k'})=
\big(a_1+1,\underbrace{2,\ldots,2}_{a_2-1},\,
a_3+2,\underbrace{2,\ldots,2}_{a_4-1},\ldots,
a_{n-1}+2,\underbrace{2,\ldots,2}_{a_{n}-1}\big).
\end{align*}

\begin{lemma}[Morier-Genoud--Ovsienko]
\label{Lemma:TriangleCount}
We have $k=k'$. Moreover, for any $l\in [1,k]$ the number $e_l$ is equal to the number of triangles in $T$ incident to vertex $l$.
\end{lemma}

The proof of the previous statement can be found in \cite[Proposition 2.1]{MO19}, after we have convinced ourselves that the combinatorial description of the sequence $(a_i)_{1\leq i\leq n}$, see \cite[Section~1.2, (1)]{MO19}, agrees with the sign sequence, compare \c{C}anak\c{c}{\i}--Schiffler~\cite[Proof of Theorem~5.3 and Figure~18]{CS13}.

There is a similar statement about vertices on the boundary of $\Sigma$ if we go from $0$ to $k+1$ in counterclockwise direction. These numbers are related to the sign sequence after reversing the choice of the internal edge in the first tile.

Note that, by construction of the snake graph $G$ and the associated generalised snake graph $\black(G)$, the upper boundary of $\black(G)$ contains exactly $k$ black vertices. For $l\in [1,k]$ the number of triangles in $T$ incident to $l$ is equal to the number of (white) vertices in $G$ adjacent to $l$. Lemma~\ref{Lemma:TriangleCount} implies the following corollary.

\begin{coro}
\label{Coro:BlackVertices}
Let $l\in [1,k]$. Then $e_l$ is equal to the number of neighbours of the black vertex $l$ in $\black(G)$.
\end{coro}

The results of Subsections~\ref{Subsec:Kasteleyn}, \ref{Subsec:WeightedAdjacency} and \ref{Subsec:BipartiteStructure} generalise to generalised snake graphs. Especially, any generalised snake graph admits a Kasteleyn weighting and all Kasteleyn weightings are equivalent. With the terms just redefined the proof of Proposition~\ref{Prop:Tridiagonal} can be easily modified to apply to generalised snake graphs $G$ where all turns are black, as can various results that follow from this Proposition. In particular, if $B$ is the weighted bipartite adjacency matrix of $G$, then $BB^T$ is a $2\times 2$ block diagonal matrices whose diagonal blocks are triadiagonal matrices. Hence, we can compute the characteristic polynomial of $BB^T$ with the recursions from Section~\ref{Sec:Recursions} or as the numerator of a continued fraction. As before, the data required for this is given by the numbers of neighbours of the black vertices on one of the two boundary path graphs.

\subsection{Perfect matchings}

\begin{theo}
\label{Theo:Matchings}
Suppose that $H$ is a generalised snake graph whose turns are all black, equipped with a Kasteleyn weighting $w$. Let $B$ be the weighted bipartite adjacency matrix of $H$ and let $B_1$ and $B_2$ be the diagonal blocks of the matrix $BB^T$ indexed by the black vertices on the upper and lower boundary. Then $\lvert\operatorname{det}(B_1)\rvert=\lvert\operatorname{det}(B_2)\rvert$ is equal to the number of perfect matchings of $H$.
\end{theo}

\begin{proof}
There is a snake graph $G$ such that $H$ is obtained from $G$ by a sequence of rotations. We consider the sign sequence $(a_1,\ldots,a_n)$ of $G$ and, without loss of generality, we may assume that $n$ is even after applying the procedure in Remark~\ref{Rem:Unique} if necessary. The numerator in \c{C}anak\c{c}{\i}--Schiffler's Formula
\begin{align*}
    [a_1,a_2,\ldots,a_n]=\frac{p}{q},
\end{align*}
see Equation~(\ref{Eqn:Expansion}), is equal to the number of perfect matchings of $G$, that is, $\PM(G)=p$. Hirzebruch and Morier-Genoud--Ovsienko's Formula, see Equation~(\ref{Eqn:TwoContinuedFractions}), implies
\begin{align*}
[[e_1,e_2,\ldots,e_k]]=\frac{p}{q}.
\end{align*}
By Corollary~\ref{Coro:BlackVertices}, $k$ is equal to the number of black vertices in one of the boundary components of $H=\black(G)$, say the one corresponding to $B_1$, and for any $l\in [1,k]$ the number $e_l$ is equal to the number of (white) vertices adjacent to the (black) vertex $l$ in $G$. The substitution $x=0$ allows us to conclude from Theorem~\ref{Theo:Numerator} that $\charpoly{B_1}(0)=\pm p$; this implies $\lvert \operatorname{det}(B_1)\rvert=\lvert \charpoly{B_1}(0)\rvert =p$. 

Notice that $p=\PM(G)=\PM(H)$ as a consequence of Lemma~\ref{Lemma:Rotate}. If $A$ denotes the weighted adjacenct matrix of $H$, then Proposition~\ref{Prop:CharPoly} implies $\operatorname{det}(A)=\pm \operatorname{det}(B_1)\operatorname{det}(B_2)$. Thanks to Temperley--Fisher and Kasteleyn's Theorem~\ref{Thm:Kasteleyn} we may rewrite the previous equations as $p^2=\pm \operatorname{det}(B_1)\operatorname{det}(B_2)$; this implies $\lvert \operatorname{det}(B_2)\rvert=p$.
\end{proof}

\begin{exam}
The graph in Figure~\ref{Fig:NeighbourCount} has the sign sequence $(+,+,-,-,+,+,-,+,-)$ which translates into a continued fraction $[2,2,2,1,1,1]=\frac{46}{19}$. On the other hand, the lower boundary of the graph has $4$ black vertices which have $3$, $2$, $4$ and $3$ neighbours, respectively, yielding the continued fraction $[[3,2,4,3]]=\frac{46}{19}$. The numerator $46$ gives us the number of perfect matchings of the graph.

On the other hand, $[[2,4,2,3,2]]=[1,1,2,2,1,2]=\frac{46}{27}$. The first expression arises when we apply Theorem~\ref{Theo:Matchings} to the the upper boundary. The second expression arises when we construct the sign sequence but declare the left edge of the first tile to be internal instead of the lower edge.
\end{exam}

\begin{rem}
Let us revisit Theorem~\ref{Theo:Matchings}. Suppose that the matrix $B_1$ is indexed by the lower boundary without loss of generality. Let $G'=(V',E')$ be the full subgraph of $G$ obtained by removing the black vertices on the upper boundary and all edges incident to them. Then the restriction of the weighting $w$ to $E'$ is a Kasteleyn weighting for $G'$ because $G'$ only has one face (the infinite face). The weighted adjacency matrix $A'$ of $G'$ is a symmetric $2\times 2$ block matrix  whose blocks are indexed by the black and white vertices of $G'$. Let $B'$ be the upper right block of $A'$, that is, $B'$ describes the adjacency between black and white vertices in $G'$. By construction $B_1=B'(B')^T$. It is possible to show that $\lvert \operatorname{det}(B_1)\rvert$ is the square of the number of perfect matchings of $G'$. To be more precise, notice that a perfect matching of $G'$ is a matching of $G$ that covers all the black vertices. Then it is possible to show that every such matching can be extended uniquely to a perfect matching of $G$.   
\end{rem}

\bigskip
\noindent
\textbf{Acknowledgments.}
\phantomsection%
\addcontentsline{toc}{section}{Acknowledgments} The article summarises an undergraduate summer research project the authors conducted at the University of Kent. J. Bradshaw received an Undergraduate Research Bursary URB 18-19 19 from the London Mathematical Society and the School of Mathematics, Statistics and Actuarial Science. P. Lampe was supported by EPSRC grant EP/M004333/1. D. Ziga is grateful to the University of Kent for financial support.

\bibliographystyle{hyperalphaabbr}
\bibliography{snake}

\end{document}